\pdfoutput=1
\documentclass[11pt]{article}
\usepackage[margin=1.0in]{geometry}

\usepackage{amsmath,amsfonts,amsthm,amssymb} 
\usepackage{bigstrut} 
\usepackage{graphicx} 
\usepackage{hhline} 
\usepackage{hyperref} 
\usepackage[numbers,sort&compress]{natbib} 
\usepackage{placeins} 
\usepackage{subcaption} 

\newtheorem{proposition}{Proposition}
\newtheorem{observation}{Observation}
\newtheorem{lemma}{Lemma}
\newtheorem{theorem}{Theorem}
\newtheorem{corollary}{Corollary}

\newcommand{\set}[2]{\left\{#1\,\middle|\,#2\right\}}

\title{A Linear Input Dependence Model for Interdependent Networks}
\author{Hemanshu Kaul\footnote{Department of Applied Mathematics, Illinois Institute of Technology, Chicago, IL 60616. E-mail: \href{mailto:kaul@iit.edu}{\tt kaul@iit.edu}} \, and Adam Rumpf\footnote{Department of Applied Mathematics, Florida Polytechnic University, Lakeland, FL 33805. E-mail: \href{mailto:arumpf@floridapoly.edu}{\tt arumpf@floridapoly.edu}}}
\date{November 18, 2021}

\begin{document}

\maketitle

\begin{abstract}
We consider a linear relaxation of a generalized minimum-cost network flow problem with binary input dependencies. In this model the flows through certain arcs are bounded by linear (or more generally, piecewise linear concave) functions of the flows through other arcs. This formulation can be used to model interrelated systems in which the components of one system require the delivery of material from another system in order to function (for example, components of a subway system may require delivery of electrical power from a separate system). We propose and study randomized rounding schemes for how this model can be used to approximate solutions to a related mixed integer linear program for modeling binary input dependencies. The introduction of side constraints prevents this problem from being solved using the well-known network simplex algorithm, however by characterizing its basis structure we develop a generalization of network simplex algorithm that can be used for its {computationally} efficient solution.
\end{abstract}


\section{Introduction}
\label{sec:intro}

As society continues its trend towards urbanization it also becomes more reliant on highly interconnected civil infrastructure systems, which include services such as transportation, electrical power, telecommunications, water, and waste management. Interconnections arise when the functionality of one system relies on {the delivery of resources from} another, for example the reliance of telephone lines and subway systems on the {delivery of power from the electrical grid in order to function}. These interdependencies make the analysis of such systems much more complicated, and cause them to become more vulnerable to both attack and natural disaster due to the possibility of \textit{cascading failures}.

A cascading failure occurs when the failure of one component of a system causes another component {which depends on the functionality of the first component} (possibly in a different system) to also fail. This may in turn cause {a third component (depending on the second component) to fail, and then a fourth component, and so on}, resulting in a chain reaction that causes damage far beyond the initial failures. This phenomenon has been the subject of a great deal of recent research incorporating a wide variety of modeling techniques \cite{ouyang2014}. For example, Kinney et al.~2005 \cite{kinney2005} studied the North American power grid by examining the effects of removing random substations, which may result in a chain reaction of substation overloads that damages a large portion of the grid. Ouyang and Due\~nas-Osorio~2011 \cite{ouyang2011} examined different strategies for designing the interfaces between different infrastructure systems in an effort to maximize resistance to random natural disaster scenarios, Ouyang and Wang~2015 \cite{ouyang2015} examined different interdependent network restoration strategies, and Ouyang~2017 \cite{ouyang2017} examined worst-case scenarios as the results of directed attacks. Zhong et al.~2019 \cite{zhong2019} studied the repair process of various interdependent networks with load-dependent cascading failures. It should also be mentioned that the interdependent network paradigm has been used to study many other types of systems in recent years, for example in Wang et al.~2014 \cite{wang2014b}, which used interdependent networks to study evolutionary games for the resolution of social dilemmas.

Many interdependent network studies focus purely on topological aspects of the networks, drawing from random graph theory and percolation theory to examine how damage can propagate through a collection of interconnected networks. Gao et al.~2012 \cite{gao2012} showed that reducing the number of interdependencies between networks could reduce their susceptibility to attack. Parandehgheibi and Modiano~2013 \cite{paran2013}, Nguyen et al.~2013 \cite{nguyen2013}, and Sen et al.~2014 \cite{sen2014} all studied the interaction between an electrical power network and a communication network by finding optimal strategies for causing as much damage propagation as possible by removing as few nodes as possible. Havlin et al.~2014 \cite{havlin2014} showed results for various attack strategies on various types of random network. Lam and Tai~2018 \cite{lam2018} used fuzzy set theory to model networks with uncertain interdependencies.

Rather than focusing purely on the network topology, many other studies of interdependent infrastructure systems have instead made use of network flows models \cite{ouyang2009}, which have long been in use for other common civil infrastructure problems like transportation design \cite{farahani2013}. Under this modeling paradigm each infrastructure system is thought of as transporting a flow of material through a collection of nodes and arcs, with the flows and the network components representing different things in different systems. For example, in part of the transportation network, flow might represent cars, arcs roads, and nodes intersections, while in the water network, flow might represent water, arcs pipelines, and nodes water treatment facilities and homes. This model allows design and recovery problems to be stated as \textit{minimum-cost network flow} (MCNF) problems \cite[pp.~4--5]{networkflows}.

Interdependencies can be introduced into the standard MCNF problem by including side constraints in addition to the usual flow conservation and and capacity constraints. Lee et al.~2007 \cite{lee2007} put forward a model for describing these relationships. This paper described a type of interdependence called \textit{input dependence}, in which a component of one infrastructure system requires delivery of resources from another system in order to function (for example, the relationship between the electrical power network and the subway network described above). Input dependence can be modeled by causing an arc in one network (the \textit{child}) to become unable to transport flow unless a demand node in another network (the \textit{parent}) receives its full flow demand. Instead of strictly enforcing that all demand be satisfied as in the standard MCNF, shortfall is allowed but is penalized by adding a penalty cost to the objective and by causing other components of the network to fail. This type of constraint has since been incorporated into other models, notably by Cavdaroglu et al., whom used it as part of their disaster recovery scheduling model \cite{cavdaroglu2010} to study a series of similar problems involving interdependent infrastructure systems \cite{cavdaroglu2013,nurre2012,nurre2014,sharkey2015}. A similar notion of input dependence has been used in other network flows-based models \cite{almoghathawi2019,holden2013}.

The modeling technique used by Lee et al.~included a binary indicator variable for each parent/child pair with logical constraints to force its value to 0 if the parent's demand was not fully met and 1 otherwise. This variable was used as a multiplier for the child arc's capacity constraint, effectively leaving its capacity unchanged if the parent's demand was satisfied and setting it to 0 otherwise. The use of binary variables can cause the resulting \textit{mixed integer linear program} (MILP) to become computationally intractable for large civil infrastructure systems.

\paragraph{{Motivation and contributions}}

The purpose of this paper is to examine the \textit{linear program} (LP) relaxation of the Lee et al.~input dependence MILP model obtained by replacing the binary indicator variables with real-valued variables on the interval $[0,1]$. In fact we study a generalization of this relaxation, which we will refer to as the \textit{minimum-cost network flow problem with linear interdependencies} (\mbox{MCNFLI}), {in which the child's capacity may be bounded by {\sl any} linear function of the parent's inflow. The main contribution of this paper is to formulate this novel model and to study its methodological and algorithmic properties. Most importantly this includes characterizing the basis of the underlying LP, which enables the development of a generalized network simplex algorithm \cite[pp.~402--460]{networkflows} that can be used for its efficient solution.}

{We also explore applications of the \mbox{MCNFLI} model as a modeling and algorithmic tool. This includes its ability to describe relationships outside the scope of the original binary input dependence model and its use as part of an approximation algorithm for the MILP model. This usage is particularly important for larger multi-staged optimization models that include an underlying interdependent flow network, such as the disaster recovery scheduling models of Cavdaroglu et al.~cited above. The solution algorithms for such models often involve an iterative approach in which the underlying network flows problem must be solved repeatedly many times over, in which case the computational savings of the LP relaxation over the MILP model may be greatly multiplied. Moreover, pure LPs possess some important mathematical properties that MILPs lack (such as strong duality \cite[p.~148]{linopt}), which can allow the LP relaxation to be used in formulating far more computationally efficient non-iterative approximation algorithms for the overall model. This application of the \mbox{MCNFLI} is explored further in Rumpf~2020 \cite{rumpfthesis}, which uses the \mbox{MCNFLI} model to develop several approximation algorithms for solving an interdependent network interdiction game.}

\paragraph{Structure of this paper}

In Section \ref{sec:formulation} we describe the formulation of the \mbox{MCNFLI} model by generalizing the LP relaxation of the binary input dependence model from Lee et al.~2007 and the motivation underlying its formulation. {Section \ref{sec:applications} discusses an application wherein the solution of the \mbox{MCNFLI} model is combined with a variety of randomized rounding schemes to build an approximate solution for the computationally intractable MILP model. Having illustrated the usefulness of solutions of the \mbox{MCNFLI} model, the remainder of the paper is dedicated to the development of an efficient solution algorithm for the \mbox{MCNFLI}.} In Section \ref{sec:structure} we explore the structure of the \mbox{MCNFLI} and derive some important theoretical results, including the main result of this paper: the characterization of the basic feasible solution. These results are then applied in Section \ref{sec:simplex} to formulate a generalized network simplex algorithm. Finally we conclude in Section \ref{sec:conclusion} with a summary of results and a discussion of future work.

\section{Model Formulation and Motivation}
\label{sec:formulation}

{The purpose of this section is to give a general formulation for the \mbox{MCNFLI} model to be used in developing the main theoretical results of this study.} The \mbox{MCNFLI} consists of an MCNF with side constraints. Since any collection of interdependent networks can be merged into a single network using artificial arcs with high cost, for the remainder of the paper {without loss of generality} we consider a {single} flow network $G = (V,E)$ with node set $V$ and arc set $E$. For each node $i \in V$ we have a supply value $b_i$ which is positive for supply nodes, negative for demand nodes, and zero for transshipment nodes. We assume that $\sum_{i \in V} b_i = 0$. For each arc $ij \in E$ we have a constant upper bound $u_{ij} \in [0,\infty]$, a unit flow cost of $c_{ij}$, and a nonnegative flow variable of $x_{ij}$.

We also have a set $I$ of interdependencies. Rather than the node-to-arc interdependencies described in Section \ref{sec:intro}, we will use arc-to-arc interdependencies in which one arc acts as the parent of a different arc. This formulation can be obtained from a node-to-arc interdependence using a transformation similar to that of the minimum-cost formulation of the maximum flow problem \mbox{\cite[pp.~430--433]{linopt}}, with a parent arc being saturated corresponding to a parent node's demand being met. Each element of $I$ is an ordered pair $(ij,kl) \in E \times E$ whose first element represents a parent arc and whose second element represents the corresponding child arc. {Without loss of generality we assume a one-to-one correspondence between parents and children, and that each arc appears in at most one interdependence, since many-to-one interdependencies, one-to-many interdependencies, and mutually interdependent pairs of arcs can all be modeled as arrangements of interdependent arcs in series. For each interdependence $(ij,kl) \in I$ we are also given constants $\alpha_{ij}^{kl}$ and $\beta_{ij}^{kl}$, satisfying $\alpha_{ij}^{kl} x_{ij} + \beta_{ij}^{kl} \ge 0$ for all $x_{ij} \in [0,u_{ij}]$,} which define the linear interdependence. The resulting model, which we will refer to as {the \textit{linear input dependence model} (\mbox{LIDM})}, is
\begin{alignat}{2}
	\label{eqn:mcnfliobjective} \min \quad& \sum_{ij \in E} c_{ij} x_{ij} \span\span \\
	\label{eqn:mcnflitransship} \mathrm{s.t.} \quad& \sum_{j : ij \in E} x_{ij} - \sum_{j : ji \in E} x_{ji} = b_i &\qquad& \forall i \in V \\
	\label{eqn:mcnflilowerflow} & 0 \le x_{ij} \le u_{ij} && \forall ij \in E \\
	\label{eqn:mcnflilinking} & x_{kl} \le \alpha_{ij}^{kl} x_{ij} + \beta_{ij}^{kl} && \forall (ij,kl) \in I
\end{alignat}

{The LIDM represents one particular formulation of the MCNFLI.} Note that (\ref{eqn:mcnfliobjective})--(\ref{eqn:mcnflilowerflow}) are exactly the objective and constraints of the MCNF. The new side constraints (\ref{eqn:mcnflilinking}) describe the interdependencies, bounding the flow through a child arc $x_{kl}$ by a linear function $\alpha_{ij}^{kl} x_{ij} + \beta_{ij}^{kl}$ of its corresponding parent arc $x_{ij}$.

The {MILP} from Lee et al.~2007, which we will refer to as {the \textit{binary input dependence model} (\mbox{BIDM})}, also takes the form of the MCNF with side constraints. It makes use of a binary linking variable $y_{ij}^{kl}$ for each interdependence $(ij,kl) \in I$. Additional constraints and slack variables force $y_{ij}^{kl} = 0$ when $x_{ij} < u_{ij}$ and $y_{ij}^{kl} = 1$ when $x_{ij} = u_{ij}$, allowing $y_{ij}^{kl}$ to act as an indicator of whether the parent is saturated. The interdependence, itself, is then enforced by a constraint of the form $x_{kl} \le u_{kl} y_{ij}^{kl}$, forcing the child's flow to be bounded by $u_{kl}$ if the parent is saturated and 0 otherwise.

The LP relaxation of {the \mbox{BIDM}} involves allowing the linking variables $y_{ij}^{kl}$ to take values in the interval $[0,1]$, rather than only the set $\{0,1\}$. Doing so allows the linking variables to be eliminated via substitution. The resulting linking constraint takes the form $x_{kl} \le \frac{u_{kl}}{u_{ij}} x_{ij}$, which is the special case of {the \mbox{LIDM}} for which $\alpha_{ij}^{kl} = \frac{u_{ij}}{u_{kl}}$ and $\beta_{ij}^{kl} = 0$ for all $(ij,kl) \in I$. For this reason, throughout the remainder of the paper we will refer to {the \mbox{LIDM}} as the LP relaxation of {\mbox{BIDM}}.

Linear input dependencies have several modeling applications of their own beyond simply approximating binary input dependencies. While binary input dependence renders a child arc completely unusable if there is any shortfall at its parent arc, linear input dependence still permits partial functionality from partial delivery, which may be more appropriate for modeling certain systems. In addition, by placing a sequence of parents and children in series, linear input dependencies as modeled in {\mbox{the LIDM}} can be used to generate a piecewise linear concave bound on the child's capacity as a function of the parent's saturation, which can model input dependencies in which the delivery of more material has diminishing returns.

\section{Approximate Mixed Integer Solutions}
\label{sec:applications}

Although, as discussed above, a linear input dependency model like {the \mbox{LIDM}} is {useful} in its own right, in this section we will investigate how solutions of {the \mbox{LIDM}} can be applied to find good approximations to the solutions of {the \mbox{BIDM}}. As {the \mbox{BIDM}} is NP-complete it is not in general computationally tractable to solve large instances of it exactly {(particularly if being used as a sub-model that must be solved repeatedly within a larger problem)}, however its linear relaxation {in the form of the \mbox{LIDM}} can be solved quickly and then those solutions can be used to build a solution for {the \mbox{BIDM}}. {In this section we describe a family of randomized rounding algorithms for obtaining a near-optimal, feasible solution to {the \mbox{BIDM}} based on its linear relaxation. We then conduct computational experiments to demonstrate that these schemes typically generate high-quality approximations within a small number of attempts for a large number of test cases, particularly for networks with relatively few interdependencies.}

\subsection{Randomized Rounding Algorithms}
\label{subsubsec:rr}

{In this section we describe the randomized rounding approximation algorithms which will be the subject of the remainder of Section \ref{sec:applications}, as well as explaining why a randomized rounding approach was chosen.} In {the \mbox{BIDM}} the binary linking variable $y_{ij}^{kl}$ for each interdependent pair $(ij,kl) \in I$ either forces the parent arc to be saturated (when $y_{ij}^{kl} = 1$) or the child arc to be unused (when $y_{ij}^{kl} = 0$). A common approximation technique for such a program with binary variables would be to solve the linear relaxation (restricting $y_{ij}^{kl}$ to $[0,1]$ instead of $\{0,1\}$) and then deterministically round the relaxed binary variables to 1 if above some threshold or 0 if below some threshold, fixing the rounded variables and then finally solving the resulting pure LP. Unfortunately there is no deterministic {threshold} that could be applied in this case that would guarantee the existence of a feasible solution. As an alternative we consider the related idea of randomized rounding algorithms.

Let $\mathbf{x^{\boldsymbol{*}}}$ be the optimal flow for the LP relaxation of {the \mbox{BIDM}}. For each interdependence $(ij,kl) \in I$ we set $y_{ij}^{kl}$ to 1 with probability $P_{ij}^{kl}$ and 0 otherwise, where $P_{ij}^{kl}$ describes the relative saturation of either the parent or the child. If the resulting binary variables define an infeasible LP then the random binary variables can be re-realized, repeating as necessary until a feasible solution is found. As it would be reasonable to base the value of $y_{ij}^{kl}$ on the saturation of either the parent or the child, we define three basic types of randomized rounding algorithm: \textit{randomized rounding based on child flow} ({\mbox{$\textsc{RR-Child}$}}) in which $P_{ij}^{kl} := x_{kl}^*/u_{kl}$, \textit{randomized rounding based on parent flow} ({\mbox{$\textsc{RR-Parent}$}}) in which $P_{ij}^{kl} := x_{ij}^*/u_{ij}$, and the control test \textit{fair randomized rounding} ({\mbox{$\textsc{RR-Fair}$}}) in which $P_{ij}^{kl} := 0.5$.

There is a danger that such a randomized rounding algorithm may not terminate regardless of how many times the random variables are realized. This can occur in certain cases where $P_{ij}^{kl}$ is allowed to equal exactly 0 or 1, as shown in the examples in Appendix \ref{app:rrfailure}. To prevent this we define modified versions of the {\mbox{$\textsc{RR-Child}$}} and {\mbox{$\textsc{RR-Parent}$}} schemes with alternate definitions for $P_{ij}^{kl}$. {For any $\epsilon \in [0,0.5]$, let \mbox{$\textsc{RR-Child}(\epsilon)$} be equivalent to \mbox{$\textsc{RR-Child}$} described above, but with its value of $P_{ij}^{kl}$ restricted to the interval $[\epsilon,1-\epsilon]$ (that is, equal to $\max\{\min\{x_{kl}^*/u_{kl},1-\epsilon\},\epsilon\}$) in order to avoid probabilities too close to 0 or 1. For the purposes of our study we will explore the special cases of \mbox{$\textsc{RR-Child}(0.00)$}, which is equivalent to \mbox{$\textsc{RR-Child}$} as described above, as well as \mbox{$\textsc{RR-Child}(0.01)$} and \mbox{$\textsc{RR-Child}(0.05)$}, which restrict $P_{ij}^{kl}$ to the intervals $[0.01,0.99]$ and $[0.05,0.95]$, respectively. We define \mbox{$\textsc{RR-Parent}(\epsilon)$} similarly to \mbox{$\textsc{RR-Child}(\epsilon)$}, with $P_{ij}^{kl} = \max\{\min\{x_{ij}^*/u_{ij},1-\epsilon\},\epsilon\}$.}

Finally we note some simple theoretical bounds. Clearly the LP relaxation of {the \mbox{BIDM}} provides a lower bound for its optimal value, and any feasible solution to {the \mbox{BIDM}} obtained from a randomized rounding scheme provides an upper bound. Unfortunately in general these bounds may not be very tight, and there exist example networks for which they are arbitrarily far from the true optimum (see Appendix \ref{app:boundexample}). For this reason theoretical bounds are not very useful in evaluating the LP relaxation and randomized rounding approximations, and so in order to learn more about the typical errors in these approximations we turn to empirical results obtained through computational trials.

{In order to constitute reasonable approximation methods we should expect that the LP relaxation and the randomized rounding schemes produce typically similar objective values to those of the {\mbox{BIDM}} solution, and that the randomized rounding schemes typically find a feasible solution in a small number of iterations. Moreover, in order to justify the utility of the LP relaxation in finding a feasible solution to the {\mbox{BIDM}} we should expect all {\mbox{$\textsc{RR-Child}$}} and {\mbox{$\textsc{RR-Parent}$}} schemes to typically produce higher-quality solutions than the {\mbox{$\textsc{RR-Fair}$}} scheme, as the {\mbox{$\textsc{RR-Child}$}} and {\mbox{$\textsc{RR-Parent}$}} schemes require an LP relaxation pre-solve while the {\mbox{$\textsc{RR-Fair}$}} scheme does not. It is also reasonable to expect that all approximation methods perform better at lower interdependence densities, as fewer interdependencies leads to fewer relaxed constraints. These expectations were substantiated through use of the computational experiments explained below.}

\subsection{Computational Trials}
\label{subsec:computational}

{The approximation methods described above were empirically tested on a large number of random interdependent networks. Each network was solved once as a MILP by treating its interdependencies as binary, once as an LP by treating its interdependencies as linear, and then once again as a MILP using each of the randomized rounding schemes (based on the LP solution). For each solution method the objective value was recorded, and for the randomized rounding schemes the number of realizations required to first achieve a feasible MILP solution was recorded. In order to evaluate the effect of the interdependence density as well as the network's size, arc density, and interdependence structure on these results, several network parameters were varied between problem sets.}

We used a modified version of NETGEN \cite{netgen}, a random MCNF problem generator, to create our test cases, adding a procedure to generate pairs of interdependent arcs alongside the standard MCNF problem parameters and used these to define instances of {the \mbox{BIDM}}, as well as a preliminary check to ensure that all problem instances were feasible. We tested two main types of problem: those in which parent arcs corresponded to demand nodes, and those in which parent arcs were distributed randomly throughout the network. We will refer to these problems as {\textsc{Structured-Type}} and {\textsc{Unstructured-Type}}, respectively. {{\textsc{Structured-Type}} problems represent a more realistic interdependence structure that might be seen in infrastructure networks, while {\textsc{Unstructured-Type}} problems assume no particular interdependence structure in order to evaluate the models' behaviors in a more general setting.} For {\textsc{Structured-Type}} problems, we applied a transformation to convert parent demand nodes into equivalent parent arcs, relaxing the demand values at these nodes as well as the supply values. The full source code for the computational trial generation can be viewed online \cite{p1code}.

For each problem type, we generated test networks on 256 and 512 nodes. For each test network the number of arcs was either 4 or 8 times the number of nodes. For {\textsc{Structured-Type}} problems we randomly selected either 2\%, 5\%, 10\%, or 15\% of the demand nodes as parents, with their corresponding children being chosen uniformly at random from the arc set. For {\textsc{Unstructured-Type}} problems we randomly selected either 1\%, 2\%, 5\%, or 10\% of the arcs within the network as parents and children.

For each combination of parameters, 60 test cases were generated. In all cases, 20\% of the nodes were sources and 20\% were sinks (in {\textsc{Structured-Type}} problems, this refers to the original number of sink nodes, before some were converted into transshipment nodes during the parent arc transformation). Each arc's cost was chosen uniformly at random from the interval $[1,100]$, with its capacity chosen from $[100,500]$ and with 100\% of {the NETGEN} skeleton arcs having maximum cost. The total supply was 10,000 per 256 nodes. The total number of test cases generated over all parameter combinations listed above was 1960.

Each test began by solving the binary input dependence MILP for the current network using CPLEX. Next the LP relaxation was solved using CPLEX, and the flow values from the optimal solution were recorded. Finally each of the seven randomized rounding schemes described above was executed in turn, using the flow values from the LP solution to determine the probability of setting each binary variable to 0 or 1. During each attempt a random binary vector was generated according to the rules of the current scheme, the resulting constraints were added to the MILP formulation, and the resulting problem was evaluated with CPLEX. If infeasible, a new random binary vector was generated from the same probability distribution and the process was repeated. If no feasible solution was found after {an arbitrarily-chosen cutoff of} 1000 attempts then the attempt was labeled as a failure and the next scheme was tested.

\subsection{Computational Results}
\label{subsubsec:results}

In this section we summarize the results of the computational trials described above. Full data tables can be found in Appendix \ref{app:tables}, {while the raw data can be viewed online \cite{p1data}}. Across all groups of trials the {\mbox{$\textsc{RR-Child}$}} and {\mbox{$\textsc{RR-Parent}$}} schemes performed so similarly that only the {\mbox{$\textsc{RR-Child}$}} and {\mbox{$\textsc{RR-Fair}$}} results will be discussed here.

\paragraph{LP Relaxation}

For the LP relaxation the primary value of interest is the relative error {(as a relative difference between the MILP and LP objective values)}, which was extremely small across all trials. Among the {\textsc{Structured-Type}} trial sets all showed a mean error of less than 0.2\% with a maximum of less than 1.7\%. More than half of the {\textsc{Structured-Type}} trial errors were exactly zero. Among the {\textsc{Unstructured-Type}} trials all mean errors were less than 1.7\% with a maximum of less than 6.3\% and a median of less than 0.2\%. {In addition, within each {\textsc{Unstructured-Type}} trial set, increasing the interdependence density for a fixed node and arc count resulted in an increased mean relative error. No clear trend was displayed among the {\textsc{Structured-Type}} trials.}

\paragraph{Randomized Rounding, {\textsc{Structured-Type}} Problems}

For the randomized rounding schemes there are two values of interest: the relative error and the number of iterations required to reach a feasible solution. For {\textsc{Structured-Type}} problems, all randomized rounding schemes except for one were able to find a feasible solution after a single iteration. The one exception occurred on a test network with 512 nodes, 4 arcs per node, and 10\% of sinks acting as interdependencies, for which {\mbox{$\textsc{RR-Child}(0.00)$}} failed to find a solution within 1000 iterations, although the remaining randomized rounding schemes all succeeded.

\begin{figure}[h]
	\centering
	{\includegraphics[height=0.22\textwidth]{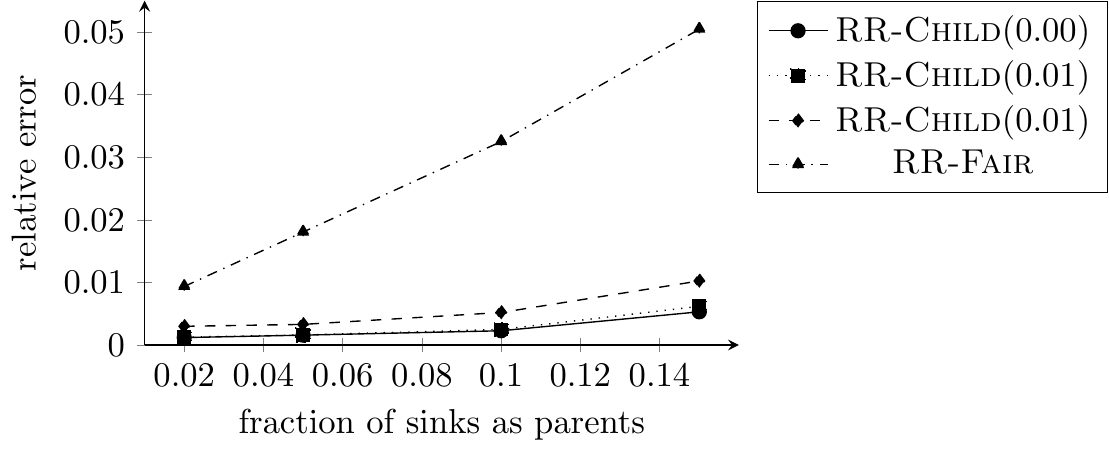}}
	\caption{Relative error in randomized rounding solutions for {\textsc{Structured-Type}} trials. Each point displays mean relative error versus interdependence density (fraction of sink nodes converted into parent arcs), calculated from the 60 {\textsc{Structured-Type}} trials on networks with 512 nodes and 4 arcs per node, ignoring the single trial for which {\mbox{$\textsc{RR-Child}(0.00)$}} failed to find a feasible solution. {All schemes display an increase in error as the interdependence density increases, though all {\mbox{$\textsc{RR-Child}$}} schemes displayed significantly smaller errors than the {\mbox{$\textsc{RR-Fair}$}} scheme.}}
	\label{fig:rrgapnode}
\end{figure}

Figure \ref{fig:rrgapnode} shows the mean relative errors for the three {\mbox{$\textsc{RR-Child}$}} schemes and the single {\mbox{$\textsc{RR-Fair}$}} scheme applied to {\textsc{Structured-Type}} problem instances, although all other {\textsc{Structured-Type}} test groups showed a similar trend. All schemes show a general trend of the mean error increasing as the interdependence density increases, with the largest errors occurring for the networks with the lowest arc density. The {\mbox{$\textsc{RR-Child}(0.00)$}} scheme achieved the smallest relative error for most trial groups followed very closely by the {\mbox{$\textsc{RR-Child}(0.01)$}} scheme, with all relative errors being less than 6.5\%. The {\mbox{$\textsc{RR-Child}(0.05)$}} trials showed larger errors, but still all less than 9.4\%. The {\mbox{$\textsc{RR-Fair}$}} trials were by far the largest, with a maximum error of approximately 18.7\%.

\paragraph{Randomized Rounding, {\textsc{Unstructured-Type}} Problems}

{\textsc{Structured-Type}} problems require relaxing some demand constraints while {\textsc{Unstructured-Type}} problems do not, making it more difficult to obtain a feasible solution. Figure \ref{subfig:rrtarc} shows the fraction of each type of trial in which each randomized rounding scheme failed to find a feasible solution within 1000 attempts. As expected, in most cases failure rates increase as the number of interdependencies increases, since this introduces additional side constraints and makes it more difficult to find a feasible solution. Other network sizes displayed less of a clear trend for small numbers of interdependencies, but in all cases the largest failure rates coincided with the largest number of interdependencies.

\begin{figure}[h]
	\centering
	\subcaptionbox{Randomized rounding scheme failure rates for {\textsc{Unstructured-Type}} trials.\label{subfig:rrtarc}}{ {\includegraphics[height=0.22\textwidth]{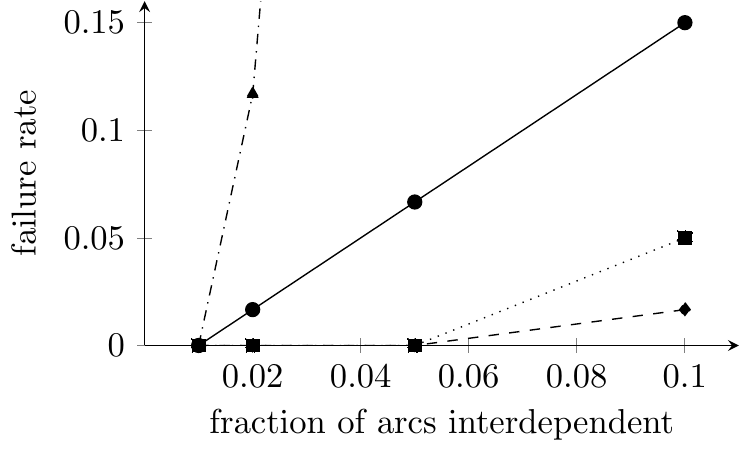}}} \hspace{0.25in}
	\subcaptionbox{Randomized rounding relative errors for {\textsc{Unstructured-Type}} trials. Statistics restricted only to successful trials.\label{subfig:rrgaparc}}{ {\includegraphics[height=0.22\textwidth]{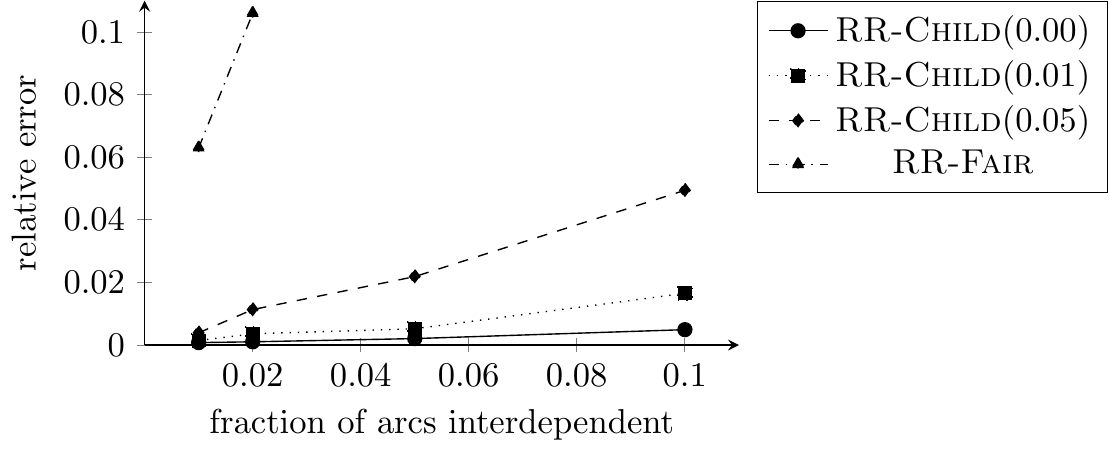}}}
	\caption{Randomized rounding schemes for {\textsc{Unstructured-Type}} trials. Statistics calculated from the 60 {\textsc{Unstructured-Type}} trials on networks with 512 nodes and 4 arcs per node. The {\mbox{$\textsc{RR-Fair}$}} series' have been cut off after the first two iterations. {Both the failure rates and the relative errors increased as the interdependence density increased, with {\mbox{$\textsc{RR-Fair}$}} performing far worse than all {\mbox{$\textsc{RR-Child}$}} schemes in both regards. Among the {\mbox{$\textsc{RR-Child}$}} schemes, those with the least randomization displayed the highest failure rates but the lowest relative errors on successful trials, indicating a tradeoff between chance of success and quality of approximation.}}
	\label{fig:rrctgap}
\end{figure}

More important is the comparison between the failure rates of the different randomized rounding schemes. In all trial groups except for one (256 nodes, 4 arcs per node, and 2\% of arcs interdependent) the {\mbox{$\textsc{RR-Fair}$}} scheme showed by far the largest failure rate, followed by {\mbox{$\textsc{RR-Child}(0.00)$}}, with {\mbox{$\textsc{RR-Child}(0.01)$}} and {\mbox{$\textsc{RR-Child}(0.05)$}} showing by far the lowest failure rates. For the trial groups in which 10\% of arcs were interdependent, all {\mbox{$\textsc{RR-Child}$}} failure rates fell below 15\% while all {\mbox{$\textsc{RR-Fair}$}} failure rates were above 95\%. Among the {\mbox{$\textsc{RR-Child}$}} schemes, as seen in Figure \ref{subfig:rrtarc}, {\mbox{$\textsc{RR-Child}(0.00)$}} failed significantly more often than {\mbox{$\textsc{RR-Child}(0.01)$}} and {\mbox{$\textsc{RR-Child}(0.05)$}}.

Figure \ref{subfig:rrgaparc} shows the mean for the relative error of each randomized rounding scheme for {\textsc{Unstructured-Type}} problems, restricted only to instances for which a feasible solution was obtained within 1000 attempts. The trend appears extremely similar to that of the {\textsc{Structured-Type}} trials. All mean errors increase with the number of interdependencies. In all cases {\mbox{$\textsc{RR-Fair}$}} produces by far the largest error, followed by {\mbox{$\textsc{RR-Child}(0.05)$}}, and then {\mbox{$\textsc{RR-Child}(0.01)$}} and finally {\mbox{$\textsc{RR-Child}(0.00)$}}. While {\mbox{$\textsc{RR-Child}(0.01)$}} and {\mbox{$\textsc{RR-Child}(0.00)$}} performed almost identically in the {\textsc{Structured-Type}} trials there is a more pronounced difference in the {\textsc{Unstructured-Type}} trials, for which {\mbox{$\textsc{RR-Child}(0.00)$}} produces clearly smaller errors than {\mbox{$\textsc{RR-Child}(0.01)$}}. Across all trial sets the largest mean error for the {\mbox{$\textsc{RR-Child}(0.05)$}} solutions was approximately 16.8\%, while for {\mbox{$\textsc{RR-Child}(0.01)$}} it was 4.1\% and for {\mbox{$\textsc{RR-Child}(0.00)$}} it was 0.8\%.

\subsection{Discussion}
\label{sub:discussion}

{The computational results of the previous section largely confirm our expectations for the approximation algorithms as explained in Section \ref{subsubsec:rr}.} The LP relaxation and the resulting randomized rounding solutions typically produce reasonable bounds for the MILP solution. The mean relative error of the LP relaxation had a mean of less than 1.7\% for all trial groups. For {\textsc{Structured-Type}} problems, with one exception the {\mbox{$\textsc{RR-Child}(0.00)$}} scheme was able to find a feasible solution at a relative error of less than 6.5\%, while {\mbox{$\textsc{RR-Child}(0.01)$}} achieved nearly the same accuracy with a better success rate. For {\textsc{Unstructured-Type}} problems the {\mbox{$\textsc{RR-Child}(0.00)$}} scheme achieved the smallest relative errors and the {\mbox{$\textsc{RR-Child}(0.01)$}} scheme achieved the smallest failure rate, but both performed well especially for the smallest trials. {In all trial sets, all {\mbox{$\textsc{RR-Child}$}} and {\mbox{$\textsc{RR-Parent}$}} schemes achieved significantly smaller errors and higher success rates than the {\mbox{$\textsc{RR-Fair}$}} scheme, which further justifies the utility in solving {an instance of the \mbox{LIDM}} while approximating the solution to {the \mbox{BIDM}}. However, these results also show that the error and failure rates of the approximation methods increase with the interdependence density. While this was expected, it also implies that the use of the randomized rounding algorithms should be limited to networks with relatively low interdependence density. That being said, in spite of the upward trend, the relative error in the LP relaxation was so small across all trial sets that it is still a useful approximation method even for networks with relatively large interdependence densities, which motivates the development of an efficient solution algorithm for {the \mbox{LIDM}}.}

The remainder of the paper will be dedicated to developing and proving the correctness of a generalized version of network simplex capable of solving an arbitrary instance of the \mbox{MCNFLI}. This problem is technically a special case of the MCNF problem with side constraints studied in Mathies and Mevert~1998 \cite{mathies1998}, but we can exploit the structure of our side constraints to define a simpler solution algorithm. Similar algorithms have been studied in Calvete~2003 \cite{calvete2003} for the general equal flow problem and in Bah\c{c}eci and Feyzio\~glu~2012 \cite{bahceci2012} for the proportional flow problem.

\section{Problem Structure and Basis}
\label{sec:structure}

In this section we present some of the notation and theoretical results that will be used throughout the rest of the paper. We prove some important properties of the matrix form that will be used in Section \ref{sec:simplex} to develop a modified network simplex algorithm, most importantly the characterization of a basis for the \mbox{MCNFLI} in Section \ref{sec:basis}. Before moving on it will be useful to express {the \mbox{LIDM}} in matrix form. To this end we introduce a slack variable $s_{ij}^{kl}$ for each $(ij,kl) \in I$ and rewrite the linking constraint (\ref{eqn:mcnflilinking}) as the equality
\begin{align}
	\label{eqn:mcnflilinking2} x_{kl} - \alpha_{ij}^{kl} x_{ij} + s_{ij}^{kl} = \beta_{ij}^{kl}
\end{align}

Let $m = |V|$, $n = |E|$, and $p = |I|$, and assume without loss of generality that $n > m > p$. Let $\mathbf{x} \in \mathbb{R}^n$ be the vector of flow variables, $\mathbf{s} \in \mathbb{R}^p$ be the vector of slack variables, $\mathbf{c} \in \mathbb{R}^n$ be the vector of unit flow costs, $\mathbf{u} \in \mathbb{R}^n$ be the vector of arc capacities, $\mathbf{b} \in \mathbb{R}^m$ be the vector of supply values, $\boldsymbol{\beta} \in \mathbb{R}^p$ be the vector of linear inequality constants, and $\mathbf{\hat{A}} \in \mathbb{R}^{m \times n}$ be the node-arc incidence matrix of $G$ \cite[p.~277]{linopt}. Let $\mathbf{\hat{Q}} \in \mathbb{R}^{p \times n}$ be a matrix with one column for each arc $ij \in E$ and one row for each interdependence $(ij,kl) \in I$. Within the row corresponding to interdependence $(ij,kl)$, element $ij$ is $-\alpha_{ij}^{kl}$, element $kl$ is $1$, and all other elements are zero. Then the block matrix form of {the \mbox{LIDM}}, which we will refer to as {the \mbox{LIDM-B}}, is
\begin{align}
	\label{eqn:matobjective} \min \quad& \mathbf{c}' \mathbf{x} \\
	\label{eqn:matequality} \mathrm{s.t.} \quad& \begin{bmatrix}
		\mathbf{\hat{A}} & \mathbf{0} \\
		\mathbf{\hat{Q}} & \mathbf{I}
	\end{bmatrix}
	\begin{bmatrix}
		\mathbf{x} \\
		\mathbf{s}
	\end{bmatrix}
	=
	\begin{bmatrix}
		\mathbf{b} \\
		\boldsymbol{\beta}
	\end{bmatrix} \\
	& \mathbf{0} \le \mathbf{x} \le \mathbf{u} \\
	\label{eqn:matlowerslack} & \mathbf{0} \le \mathbf{s}
\end{align}

where $\mathbf{c}'$ denotes the transpose of $\mathbf{c}$, $\mathbf{I}$ is a $p \times p$ identity matrix, and $\mathbf{0}$ is a matrix or vector of zeros of the appropriate dimensions. Let $\mathbf{A} \in \mathbb{R}^{(m+p) \times (n+p)}$ be the matrix on the lefthand side of (\ref{eqn:matequality}), which constitutes the constraint matrix of {the \mbox{LIDM-B}}.

\subsection{Terminology}
\label{subsec:notation}

{In this section we define some terminology that will be used throughout the remainder of the paper.} For brevity we will use the terms ``arc'', ``flow variable'', and ``column'' interchangeably to refer to an arc, an arc's associated flow variable, and a flow variable's associated column of $\mathbf{A}$. Similarly we will use the terms ``node'', ``flow conservation constraint'', and ``row'' interchangeably to refer to a node, a node's associated flow conservation constraint, and a constraint's associated row of $\mathbf{A}$.

During each iteration of simplex, let $B$ denote the set of basic variables. For such $B$, let $\mathbf{B}$ be the submatrix of $\mathbf{A}$ containing only basic columns. Let $L$ and $U$ be the sets of nonbasic variables at their lower or upper bound, respectively. Let a variable be called \textit{interdependent} if it is involved in any interdependence as either a parent arc, child arc, or slack variable, and \textit{independent} otherwise. For each interdependence $(ij,kl) \in I$, we say that flow variables $x_{ij}$ and $x_{kl}$ and slack variable $s_{ij}^{kl}$ are all \textit{linked} to each other. We call arcs $ij$ and $kl$ \textit{partners} of each other. For any interdependent arc $ij$, let $\hat{a}_{ij}$ be defined as the coefficient of variable $x_{ij}$ in its corresponding linking constraint (\ref{eqn:mcnflilinking2}), which is $-\alpha_{ij}^{kl}$ if $ij$ is a parent and 1 if it is a child.

\subsection{Basis Characterization}
\label{sec:basis}

{In this section we prove the main result of the paper, the characterization of the basis of {the \mbox{LIDM-B}}, which will require first establishing some preliminary observations regarding the problem's structure.} A basis of {the \mbox{LIDM-B}} consists of a set $B$ of basic variables whose corresponding columns of $\mathbf{A}$ are full-rank, plus a partition of all nonbasic variables into sets $L$ and $U$. We begin by observing the rank of $\mathbf{A}$.

\begin{observation}
	\label{obs:rankofa} The coefficient matrix $\mathbf{A}$ of $\textup{{the \mbox{LIDM-B}}}$ has rank $m+p-1$.
\end{observation}

The proof of this follows from the fact that, as an $m \times n$ incidence matrix, submatrix $\mathbf{\hat{A}}$ has rank $m-1$ \cite[p.~280]{linopt} while the $p \times p$ identity matrix $\mathbf{I}$ in the lower right has rank $p$, and the block structure of $\mathbf{A}$ makes it clear that the last $p$ rows are linearly independent when taken together with any full-rank subset of the first $m$ rows. This tells us that our basis must always contain $m+p-1$ basic variables.

Our next goal is to find a graph-based characterization of this basis analogous to the spanning tree basis of the standard MCNF. We assume without loss of generality that $G$ contains at least one spanning tree made up entirely of independent arcs. As in the MCNF, the subgraph of basic arcs must contain only a single component.

\begin{observation}
	\label{obs:onecomponent} Given any basis $B$, the subgraph of $G$ containing only the arcs in $B$ must be connected.
\end{observation}

The proof of this relies on the fact that the basis matrix $\mathbf{B}$ can be rearranged into a block form with one block for each component of the basic subgraph. Each block is an incidence matrix and thus rank-deficient by 1, meaning that adding all rows corresponding to one component results in a zero row. Since $\mathbf{B}$ is rank-deficient by only 1, only one zero row should be possible, which implies that only one component is possible.

Unlike the standard MCNF it is possible for the \mbox{MCNFLI} basis to contain cycles of arcs, however no such cycle can consist entirely of independent arcs.

\begin{observation}
	\label{obs:nocycles} Given any basis $B$, the subgraph of $G$ containing only the independent arcs in $B$ must be free of cycles.
\end{observation}

The proof of this is almost identical to that of the MCNF problem \cite[p.~283]{linopt}. Observations \ref{obs:onecomponent} and \ref{obs:nocycles}, combined, imply that the subgraph of basic independent arcs must form a spanning forest of $G$. Let a spanning forest with $k$ components be referred to as a \textit{$k$-spanning forest}.

For the remainder of this paper we will let $r$ represent the number of basic interdependent variables, which could include a combination of interdependent arcs and slack variables. Note that it is always the case that $r \ge p$ or else Observation \ref{obs:nocycles} would be violated. This allows us to define the structure of the basic independent arcs.

\begin{lemma}
	\label{lemma:forest} The basic independent arcs must form an $(r-p+1)$-spanning forest of $G$.
\end{lemma}

This follows from Observations \ref{obs:onecomponent} and \ref{obs:nocycles}. Then the collection of basic variables consists of a spanning forest of independent arcs, plus possibly some interdependent arcs linking the components of the forest, plus possibly some slack variables. For the remainder of the paper we will use $F = \{T_1,\dots,T_{r-p+1}\}$ to refer to the spanning $(r-p+1)$-forest of $G$ formed by the basic independent arcs, where $T_h, h=1,\dots,r-p+1$ are the subgraphs of $G$ that form the components of the forest.

Define matrix $\mathbf{D} \in \mathbb{R}^{r \times r}$ as
\begin{align}
	\label{eqn:ddef} \mathbf{D} := \left[ \begin{array}{c c c|c}
		\delta_1^1 & \cdots & \delta_{r_a}^1 & \\
		\vdots & \ddots & \vdots & \\
		\delta_1^{r-p} & \cdots & \delta_{r_a}^{r-p} & \\
		\hline & \mathbf{Q} & & \mathbf{\hat{I}}
	\end{array} \right]
\end{align}

where $r_a$ is the number of basic interdependent arcs, $r_s$ is the number of basic slack variables, $\mathbf{\hat{I}} \in \mathbb{R}^{p \times r_s}$ and $\mathbf{Q} \in \mathbb{R}^{p \times r_a}$ are the submatrices of $\mathbf{\hat{I}}$ and $\mathbf{\hat{Q}}$ containing only basic columns, and $\delta_{ij}^h$ is defined as
\begin{align}
	\label{eqn:deltadef} \delta_{ij}^h := \left\{
	\begin{array}{r l}
		1 & \text{if } i \in T_h \text{ and } j \notin T_h \\
		-1 & \text{if } i \notin T_h \text{ and } j \in T_h \\
		0 & \text{otherwise}
	\end{array}
	\right. \qquad \forall ij \in E, \forall h=1,\dots,r-p+1
\end{align}

This matrix proves to be the key for determining the rank of $\mathbf{B}$.

\begin{lemma}
	\label{lemma:characterization} The rank of the basis matrix $\mathbf{B}$ is $m+p-1$ if and only if matrix $\mathbf{D}$ as defined in (\ref{eqn:ddef}) has rank $r$.
\end{lemma}

\begin{proof}
	We can reorder the columns and rows of any basis matrix $\mathbf{B}$ to obtain the block structure
	\begin{align}
		\label{eqn:bstructure} \mathbf{B} = \left[ \begin{array}{c|c|c|c|c|c}
			\mathbf{T}_1 & & & & \mathbf{B}_1 & \\
			\hline & \mathbf{T}_2 & & & \mathbf{B}_2 & \\
			\hline & & \ddots & & \vdots & \\
			\hline & & & \mathbf{T}_{r-p+1} & \mathbf{B}_{r-p+1} & \\
			\hline & & & & \mathbf{Q} & \mathbf{\hat{I}}
		\end{array} \right]
	\end{align}
	
	where $\mathbf{T}_h$ is the incidence matrix of $T_h$ and $\mathbf{B}_h$ is the submatrix of $\mathbf{B}$ whose columns correspond to interdependent and slack variables and whose rows correspond to the nodes of $T_h$. For notational convenience, for the remainder of the proof we will refer to arcs by a single index rather than the indices of both of their endpoints. Let the basic interdependent arcs be labeled with indices $e=1,\dots,r_a$ corresponding to the order of the columns in each submatrix $\mathbf{B}_h$ from (\ref{eqn:bstructure}).
	
	Consider one of the submatrices $\mathbf{T}_h$. As the incidence matrix of a tree, it has exactly one more row than it has columns and it is rank-deficient by exactly 1 \cite[p.~280]{linopt}. Summing all rows of $\mathbf{T}_h$ yields the zero vector. Transform $\mathbf{B}$ into an equivalent matrix by adding all rows corresponding to $\mathbf{T}_h$ to the first row of $\mathbf{T}_h$, for each tree $h$. After this addition, the element corresponding to the first row of $\mathbf{T}_h$ in column $e$ will become $\sum_{k \in T_h} a_e^k$, which is exactly $\delta_e^h$. The result is
	\begin{align}
		\label{eqn:btransformed} \mathbf{B} \sim \left[ \begin{array}{c|c|c|c||c c c||c}
			\mathbf{0}' & & & & \delta_1^1 & \cdots & \delta_{r_a}^1 & \bigstrut \\
			\hline \mathbf{\tilde{T}}_1 & & & & & \mathbf{\tilde{B}}_1 & & \bigstrut \\
			\hhline{=|=|=|=#===#=} & \mathbf{0}' & & & \delta_1^2 & \cdots & \delta_{r_a}^2 \bigstrut \\
			\hline & \mathbf{\tilde{T}}_2 & & & & \mathbf{\tilde{B}}_2 & & \bigstrut \\
			\hhline{=|=|=|=#===#=} & & \ddots & & & \vdots & & \bigstrut \\
			\hhline{=|=|=|=#===#=} & & & \mathbf{0}' & \delta_1^{r-p+1} & \cdots & \delta_{r_a}^{r-p+1} & \bigstrut \\
			\hline & & & \mathbf{\tilde{T}}_{r-p+1} & & \mathbf{\tilde{B}}_{r-p+1} & & \bigstrut \\
			\hhline{=|=|=|=#===#=} & & & & & \mathbf{Q} & & \mathbf{\hat{I}} \bigstrut
		\end{array} \right]
	\end{align}
	
	where $\mathbf{\tilde{T}}_h$ and $\mathbf{\tilde{B}}_h$ are exactly $\mathbf{T}_h$ and $\mathbf{B}_h$, respectively, with their first row removed. All matrices $\mathbf{\tilde{T}}_h$ are full-rank \cite[p.~281]{linopt}, and since they occupy disjoint sets of columns, the rows associated with $\mathbf{\tilde{T}}_h$ must all be linearly independent. Their total rank when taken together is $m-r+p-1$ ($m$ nodes, minus one deleted row for each of the $r-p+1$ trees). Then $\mathbf{B}$ has rank $m+p-1$ exactly when all remaining rows have rank $r$, and are linearly independent when taken together with the rows of $\mathbf{\tilde{T}}_h, h=1,\dots,r-p+1$. Since the first row associated with each matrix $\mathbf{T}_h$ is zero in (\ref{eqn:btransformed}) we need only examine the remaining rows for columns 1 through $r$. The submatrix defined by this subset is
	\begin{align}
		\label{eqn:bsub} \left[ \begin{array}{c c c|c}
			\delta_1^1 & \cdots & \delta_{r_a}^1 & \\
			\vdots & \ddots & \vdots & \\
			\delta_1^{r-p+1} & \cdots & \delta_{r_a}^{r-p+1} & \\
			\hline & \mathbf{Q} & & \mathbf{\hat{I}} \bigstrut
		\end{array} \right]
	\end{align}
	
	This dimensions of this matrix are $(r+1) \times r$, since it has one column for each basic interdependent variable and one row for each of the $r-p+1$ trees in the spanning forest, plus one for each of the $p$ interdependencies. Consider the effects of summing rows 1 through $r-p+1$. In column $e$, from our earlier conclusion that $\delta_e^h = \sum_{k \in T_h} a_e^k$, this is $\sum_{h=1}^{r-p+1} \delta_e^h = \sum_{h=1}^{r-p+1} \sum_{k \in T_h} a_e^k = \sum_{k \in V} a_e^k$, which is zero since all columns of an incidence matrix sum to zero. Because of this we may drop {any arbitrary} row without affecting the rank of this matrix. {Dropping the row corresponding to tree $T_{r-p+1}$} gives exactly the $r \times r$ matrix $\mathbf{D}$ defined earlier in (\ref{eqn:ddef}).
\end{proof}

Adopting the terminology of \cite{ahuja1999,bahceci2012,calvete2003}, we refer to an $(r-p+1)$-spanning forest of $G$ as a \textit{good $(r-p+1)$-forest} with respect to a set $S$ of $r$ non-independent variables if $\mathbf{D}$ is full-rank. Then the preceding results allow us to finally characterize the basis of {the \mbox{LIDM-B}}.

\begin{theorem}
	\label{thm:basisdef} A basis of \textup{{the \mbox{LIDM-B}}} consists of a set $B$ of $m+p-1$ basic variables, containing a good $(r-p+1)$-spanning forest of $G$ verified by a set $S$ of $r$ interdependent variables, plus an assignment of all nonbasic variables to a category $L$ or $U$, containing variables at their lower or upper bounds, respectively.
\end{theorem}

The proof of Theorem \ref{thm:basisdef} follows directly from Lemma \ref{lemma:characterization} and the definition of a basis. It should also be noted that the basis characterization presented in Theorem \ref{thm:basisdef} can be modified to depend on only a submatrix of $\mathbf{D}$, as shown in Appendix \ref{app:dhat}. {In either case the spanning forest element of the {the \mbox{LIDM-B}} basis makes it structurally similar to that of the standard MCNF basis. This enables us to develop a generalized network simplex algorithm for its solution, as will be explained in the next section.}

\section{Modified Network Simplex}
\label{sec:simplex}

{The network simplex algorithm is a well-known, computationally efficient, exact solution algorithm for the network flows problems \cite[pp.~402--421]{networkflows}. It is a specialized version of the simplex algorithm that takes advantage of the network structure of the MCNF, however the introduction of side constraints in the LIDM-B prevents network simplex from being applied directly. Fortunately the basis structure derived in the previous section is so similar to the spanning tree basis of a standard MCNF problem that the network simplex algorithm can be generalized to develop an efficient solution algorithm for the \mbox{MCNFLI}.}

This section will focus primarily on the elements of network simplex that need to be modified. {Our overall methodology in developing this generalization follows the general approach used by Calvete~2003 \cite{calvete2003} and by Bah\c{c}eci and Feyzio\~glu~2012 \cite{bahceci2012} in generalizing network simplex for the equal and proportional flows problems, respectively, though these problems are not equivalent to the MCNFLI and so their specific results cannot be directly applied. Among other changes, solving the LIDM-B requires further procedures to accommodate the non-network slack variables $s_{ij}^{kl}$ as well as casework to accommodate different types of basis changes.} In Section \ref{subsec:reducedcosts} we describe the modified reduced cost calculation procedure (which searches for candidates to enter the basis). In Section \ref{subsec:changeofbasis} we describe the modified change of basis procedure (consisting of pivoting to bring a candidate into the basis while another variable leaves). We conclude in Section \ref{subsec:complexity} with an analysis of of their computational complexity. See also Appendix \ref{app:byhand} for an illustrative numerical example of the algorithm carried out on a test network.

\subsection{Reduced Cost Calculation}
\label{subsec:reducedcosts}

The vector of reduced costs $\mathbf{c^{\boldsymbol{\pi}}} \in \mathbb{R}^{n+p}$ is defined by $(\mathbf{c^{\boldsymbol{\pi}}})' := \mathbf{c}' - \boldsymbol{\pi}' \mathbf{A}$, where $\boldsymbol{\pi} \in \mathbb{R}^{m+p}$ is the \textit{potential vector} \cite[p.~84]{linopt}. There is one reduced cost $c_{ij}^\pi$ for each flow variable $x_{ij}$ and one $(c^\pi)_{ij}^{kl}$ for each slack variable $s_{ij}^{kl}$. There is one potential $\pi_i$ for reach node $i \in V$ and one $\pi_{ij}^{kl}$ for each interdependence $(ij,kl) \in I$. Given our definitions for $\mathbf{c^{\boldsymbol{\pi}}}$ and $\mathbf{A}$, we have the relationships
\begin{alignat}{2}
	\label{eqn:rcindependent} c_{ij}^\pi &= c_{ij} - (\pi_i - \pi_j) &\qquad& \text{if $ij$ is independent} \\
	\label{eqn:rcparent} c_{ij}^\pi &= c_{ij} - (\pi_i - \pi_j - \alpha_{ij}^{kl} \pi_{ij}^{kl}) && \text{if $(ij,kl) \in I$} \\
	\label{eqn:rcchild} c_{ij}^\pi &= c_{ij} - (\pi_i - \pi_j + \pi_{kl}^{ij}) && \text{if $(kl,ij) \in I$} \\
	\label{eqn:rcslack} (c^\pi)_{ij}^{kl} &= -\pi_{ij}^{kl} && \text{if $(ij,kl) \in I$}
\end{alignat}

The reduced cost of a basic variable is zero, allowing us to use equations (\ref{eqn:rcindependent})--(\ref{eqn:rcslack}) to solve for the potentials by adopting a two-step strategy based on that of Calvete~2003 \cite{calvete2003} and Bah\c{c}eci and Feyzio\~glu~2012 \cite{bahceci2012}. First we make an initial guess $\boldsymbol{\pi}$ for the potential vector. To do this, within each tree $T_h$ of $F$, arbitrarily select a root and arbitrarily assign it a potential. Assign a potential of zero to all interdependencies of basic slack variables and arbitrarily assign a potential to all remaining interdependencies. Then equations (\ref{eqn:rcindependent}) can be used to solve for all remaining node potentials within each tree in order from root to leaf, since $c_{ij}^\pi = 0$ for all $ij \in T_h$.

Second, we test whether the guessed potentials were correct by testing whether they produce reduced costs of zero for all basic variables when substituted into (\ref{eqn:rcindependent})--(\ref{eqn:rcslack}). If so, then $\boldsymbol{\pi}$ is the correct potential vector and we can move on. If not, then we must calculate a ``corrected potential'' vector $\mathbf{\tilde{\boldsymbol{\pi}}}$ defined by
\begin{alignat}{2}
	\label{eqn:corrected1} \tilde{\pi}_i &:= \pi_i + \sigma_h &\qquad& \forall i \in T_h, \forall h=1,\dots,r-p \\
	\label{eqn:corrected3} \tilde{\pi}_{ij}^{kl} &:= \sigma_{ij}^{kl} && \forall (ij,kl) \in I
\end{alignat}

The vector $\boldsymbol{\sigma}$ of correction terms is defined as the solution of the linear system $\mathbf{D}' \boldsymbol{\sigma} = \mathbf{c^{\boldsymbol{\pi}}}$, which includes the reduced cost vector $\mathbf{c^{\boldsymbol{\pi}}}$ resulting from the initial guess $\boldsymbol{\pi}$.

\begin{proposition}
	\label{prop:dsigma} If $\boldsymbol{\sigma}$ is the solution to the system $\mathbf{D}' \boldsymbol{\sigma} = \mathbf{c^{\boldsymbol{\pi}}}$, then the vector of corrected potentials $\mathbf{\tilde{\boldsymbol{\pi}}}$ calculated by equations (\ref{eqn:corrected1})--(\ref{eqn:corrected3}) will yield a reduced cost of 0 for all basic variables calculated by equations (\ref{eqn:rcindependent})--(\ref{eqn:rcslack}).
\end{proposition}

\begin{proof}
	Let $\boldsymbol{\sigma}$ be the solution to $\mathbf{D}' \boldsymbol{\sigma} = \mathbf{c^{\boldsymbol{\pi}}}$. From the first part of the two-step strategy described above it is clear that $\boldsymbol{\pi}$ results in $c_{ij}^\pi = 0$ for all $ij \in T_h$, and since the same constant offset $\sigma_h$ is applied to all node potentials in $T_h$, this remains the case after applying the correction terms. From the structure of $\mathbf{D}$ we find that the row of system $\mathbf{D}' \boldsymbol{\sigma} = \mathbf{c^{\boldsymbol{\pi}}}$ corresponding to interdependence $(ij,kl)$ is simply the equation $\sigma_{ij}^{kl} = (c^\pi)_{ij}^{kl}$. For a basic slack variable $s_{ij}^{kl}$ we have $(c^\pi)_{ij}^{kl} = 0$, in which case the corrected potential and thus the new reduced cost are all also zero.
	
	All that remains is to show that the corrected potentials also result in zero reduced cost for the basic interdependent arcs. For convenience define $\sigma_{r-p+1} \equiv 0$. For any basic parent arc $ij$ with $(ij,kl) \in I$, where $i \in T_h$ and $j \in T_g$, substituting the corrected potentials into the righthand side of (\ref{eqn:rcparent}) and simplifying gives $c_{ij}^\pi - (\sigma_h - \sigma_g - \alpha_{ij}^{kl} \sigma_{ij}^{kl})$, which is zero if and only if $\sigma_h - \sigma_g - \alpha_{ij}^{kl} \sigma_{ij}^{kl} = c_{ij}^\pi$. To show this, the row of $\mathbf{D}' \boldsymbol{\sigma}$ corresponding to arc $ij$ is $\sum_{q=1}^{r-p} \delta_{ij}^q \sigma_q - \alpha_{ij}^{kl} \sigma_{ij}^{kl}$. If $h \ne g$, then $\delta_{ij}^h$ is 1 for $q=h$, $\delta_{ij}^g = -1$, and $\delta_{ij}^q = 0$ for all other $q$, resulting in $\sum_{q=1}^{r-p} \delta_{ij}^q \sigma_q - \alpha_{ij}^{kl} \sigma_{ij}^{kl} = \sigma_h - \sigma_g - \alpha_{ij}^{kl} \sigma_{ij}^{kl}$. The same result holds even if $h = g$, since in this case $\delta_{ij}^h = \delta_{ij}^g = 0$ and $\sigma_h - \sigma_g = 0$. In either case, this row of $\mathbf{D}' \boldsymbol{\sigma}$ is set equal to $c_{ij}^\pi$, meaning that $\boldsymbol{\sigma}$ must satisfy $\sigma_h - \sigma_g - \alpha_{ij}^{kl} \sigma_{ij}^{kl} = c_{ij}^\pi$. As noted above, this implies that $\mathbf{\tilde{\boldsymbol{\pi}}}$ produces $c_{ij}^\pi = 0$ when substituted into (\ref{eqn:rcparent}). A similar argument holds for the basic child arcs, thus the overall claim holds.
\end{proof}

At the end of the two-step process the resulting corrected potential vector $\mathbf{\tilde{\boldsymbol{\pi}}}$ can be applied to equations (\ref{eqn:rcindependent})--(\ref{eqn:rcslack}) to compute the reduced cost of any nonbasic variable. As noted at the end of Section \ref{sec:basis}, there is a way to characterize the basis using only a submatrix of $\mathbf{D}$. This reduced version of $\mathbf{D}$ also gives rise to a reduced version of the above algorithm explained in Appendix \ref{app:dhat}. In addition, note that the full version of the above algorithm only needs to be used during the initial iteration. Between consecutive iterations of simplex, within most components of $F$ the previous potential vector will still satisfy (\ref{eqn:rcindependent}) after the change of basis. As a result the previous potential vector $\boldsymbol{\pi}$ can be used to calculate the correction vector $\boldsymbol{\sigma}$ without the need to recalculate all node potentials within each component of $F$.

To explain, if a component $T_h$ of $F$ remains unchanged between iterations (meaning that no arc within $T_h$ leaves the basis and no independent arc incident to a node in $T_h$ enters the basis), or if $T_h$ loses an arc and splits into two separate components, then all previous node potentials within $T_h$ still satisfy (\ref{eqn:rcindependent}). Only if an independent arc enters the basis, causing two components of $F$ to merge, can one of these equations become violated. In this event equations (\ref{eqn:rcindependent}) can be satisfied by increasing all node potentials within one of the merged components by an appropriate constant. If independent arc $ij$ with endpoints in separate components $T_h$ and $T_g$ of $F$, respectively, enters the basis, then either $(\pi_i - c_{ij}) - \pi_j$ can be added to all node potentials in $T_g$ or $(\pi_j + c_{ij}) - \pi_i$ can be added to all node potentials in $T_h$. All basic slack potentials can then be set to 0 as before, after which the correction vector can be calculated and used it to calculate the corrected potential vector.

In any case, if all reduced costs of variables in $L$ are nonnegative and all reduced costs of variables in $U$ are nonpositive, then the current basic solution is optimal and the algorithm may terminate. Otherwise we choose some variable in $L$ with a negative reduced cost or in $U$ with a positive reduced cost to enter the basis and conduct change of basis.

\subsection{Change of Basis}
\label{subsec:changeofbasis}

After deciding on a variable to enter the basis we conduct pivoting and update the solution vector. Before describing this process it will be helpful to first discuss how to calculate the values of the basic variables given a particular basis. {This may be a necessary step during the algorithm's first iteration starting from an initial basis, and it will also introduce some of the notation and results required to define the change of basis procedure in Section \ref{subsubsec:changeofbasis}.}

\subsubsection{Computing the Values of the Basic Variables}
\label{subsubsec:basicvalues}

{Our goal in this section is to derive a procedure to calculate the values of the basic variables given a valid basis $(B,L,U)$. For any such basis the} variables in $L$ and $U$ are fixed at their lower or upper bound, respectively, and the variables in $B$ must satisfy transshipment constraints (\ref{eqn:mcnflitransship}) and linking constraints (\ref{eqn:mcnflilinking2}). For $h=1,\dots,r-p+1$, define $E_h^- := \set{ij}{x_{ij} \in U, i \in T_h, j \notin T_h}$ and $E_h^+ := \set{ij}{x_{ij} \in U, i \notin T_h, j \in T_h}$.

Adopting the terminology of Calvete~2003 \cite{calvete2003}, for each tree $T_h$ of $F$ we define a \textit{net requirement} $b(T_h) := \sum_{i \in T_h} b_i + \sum_{ij \in E_h^+} u_{ij} - \sum_{ij \in E_h^-} u_{ij}$, which represents the net inflow to $T_h$ through nonbasic arcs and supply values. Let $\mathbf{\tilde{x}} \in \mathbb{R}^r$ be the vector of basic interdependent flow variables and all slack variables, arranged in the same order as the columns of $\mathbf{D}$. For interdependence $(ij,kl) \in I$ corresponding to row $t$ of $\mathbf{D}$, define its net requirement $b(t)$ as
\begin{align}
	\label{eqn:btdef} b(t) := \left\{
	\begin{array}{l l}
		\beta_{ij}^{kl} & \text{if } x_{ij}, x_{kl} \notin U \\
		\beta_{ij}^{kl} + \alpha_{ij}^{kl} u_{ij} & \text{if } x_{ij} \in U, x_{kl} \notin U \\
		\beta_{ij}^{kl} - u_{kl} & \text{if } x_{ij} \notin U, x_{kl} \in U \\
		\beta_{ij}^{kl} + \alpha_{ij}^{kl} u_{ij} - u_{kl} & \text{if } x_{ij}, x_{kl} \in U
	\end{array}
	\right.
\end{align}

Let $\mathbf{\tilde{b}} \in \mathbb{R}^r$ be a vector of the form $\left[ \begin{array}{c c c|c c c} b(T_1) & \hspace{-0.33em} \cdots \hspace{-0.33em} & b(T_{r-p}) & b(1) & \hspace{-0.33em} \cdots \hspace{-0.33em} & b(p) \end{array} \right]'$, containing all tree net requirement values $b(T_h)$ for $h=1,\dots,r-p$ followed by all interdependence net requirement values $b(t)$ for $t=1,\dots,p$. These net requirement values can be used to compute the values of the basic variables.

\begin{proposition}
	\label{prop:basicvar} Given a basis $(B,L,U)$, the values of the basic interdependent variables and the slack variables are given by the solution $\mathbf{\tilde{x}}$ of the system $\mathbf{D} \mathbf{\tilde{x}}= \mathbf{\tilde{b}}$.
\end{proposition}

\begin{proof}
	Suppose that $\mathbf{\tilde{x}}$ satisfies $\mathbf{D} \mathbf{\tilde{x}} = \mathbf{\tilde{b}}$. Our goal is to show that the values in $\mathbf{\tilde{x}}$ satisfy constraints (\ref{eqn:mcnflitransship}) and (\ref{eqn:mcnflilinking2}). Let $E_{\mathrm{ind}}$ be the set of basic independent arcs, $E_{\mathrm{int}}$ be the set of basic interdependent arcs, $\tilde{E}_{\mathrm{ind}}$ be the set of nonbasic independent arcs at their upper bound, and $\tilde{E}_{\mathrm{int}}$ be the set of nonbasic interdependent arcs at their upper bound. Define
	\begin{align}
		\label{eqn:bhatdef} \hat{b}_i := b_i - \sum_{j : ij \in \tilde{E}_{\mathrm{ind}}} u_{ij} + \sum_{j : ji \in \tilde{E}_{\mathrm{ind}}} u_{ji} - \sum_{j : ij \in \tilde{E}_{\mathrm{int}}} u_{ij} + \sum_{j : ji \in \tilde{E}_{\mathrm{int}}} u_{ji} \qquad \forall i \in V
	\end{align}
	
	This is the outflow from node $i$ required to satisfy (\ref{eqn:mcnflitransship}) after all nonbasic arc flows are taken into account. The basic flow variables must carry this amount of net outflow from node $i$, giving
	\begin{align}
		\label{eqn:bhat} \sum_{j : ij \in E_{\mathrm{ind}}} x_{ij} - \sum_{j : ji \in E_{\mathrm{ind}}} x_{ji} + \sum_{j : ij \in E_{\mathrm{int}}} x_{ij} - \sum_{j : ji \in E_{\mathrm{int}}} x_{ji} = \hat{b}_i \qquad \forall i \in V
	\end{align}
	
	For any tree $T_h$, consider the result of summing equations (\ref{eqn:bhat}) over all $i \in T_h$. All basic independent arcs have either both or neither endpoints in $T_h$, thus for all $ij \in E_{\mathrm{ind}} \cap T_h$ the terms $x_{ij}$ and $-x_{ij}$ will each appear in one equation, causing the two to cancel. The same cancellation occurs when summing equations (\ref{eqn:bhatdef}) over all $i \in T_h$. Equating the two gives
	\begin{align}
		\label{eqn:bhatsum} \sum_{i \in T_h} \left[ \sum_{j : ij \in E_{\mathrm{int}}} x_{ij} - \sum_{j : ji \in E_{\mathrm{int}}} x_{ji} \right] = \sum_{i \in T_h} \left[ b_i - \sum_{j : ij \in \tilde{E}_{\mathrm{int}}} u_{ij} + \sum_{j : ji \in \tilde{E}_{\mathrm{int}}} u_{ji} \right] \qquad \forall h=1,\dots,r-p
	\end{align}
	
	Both sides of (\ref{eqn:bhatsum}) represent the net outflow of an entire tree $T_h$ after taking nonbasic flows into account. The only arcs capable of transporting a variable amount of flow between trees are the basic interdependent arcs, and so equations (\ref{eqn:bhatsum}) define a necessary and sufficient condition for the basic interdependent flow variables to satisfy (\ref{eqn:mcnflitransship}). On the lefthand side, for any arc $ij \in E_{\mathrm{int}}$, the term $x_{ij}$ is present if $ij$ exits $T_h$, the term $-x_{ij}$ is present if $ij$ enters $T_h$, and otherwise no term of the form $x_{ij}$ is present. Then the lefthand side is equal to $\sum_{ij \in T_h} \delta_{ij}^h x_{ij}$. Similarly, on the righthand side, for any arc $ij \in \tilde{E}_{\mathrm{int}}$, the term $u_{ij}$ is present if $ij$ exits $T_h$, the term $-u_{ij}$ is present if $ij$ enters $T_h$, and otherwise no term of the form $u_{ij}$ is present. Then the righthand side is equal to $\sum_{i \in T_h} b_i - \sum_{ij \in E_h^-} u_{ij} + \sum_{ij \in E_h^+} u_{ij}$, which is exactly the definition of $b(T_h)$. Combining these two gives $\sum_{ij \in T_h} \delta_{ij}^h x_{ij} = b(T_h)$, which is row $h$ of system $\mathbf{D} \mathbf{\tilde{x}}= \mathbf{\tilde{b}}$.
	
	All that remains is to show that the slack variables in $\mathbf{\tilde{x}}$ satisfy constraints (\ref{eqn:mcnflilinking2}). Each linking constraint has the form $x_{kl} - \alpha_{ij}^{kl} x_{ij} + s_{ij}^{kl} = \beta_{ij}^{kl}$. If either flow variable is nonbasic then its value is either zero or its upper bound (depending on whether it is in $L$ or $U$), and placing the resulting constant value on the righthand side produces the net requirement $b(t)$ as defined in (\ref{eqn:btdef}). Then linking constraint $t$ is exactly row $t$ of the system $\mathbf{D} \mathbf{\tilde{x}}= \mathbf{\tilde{b}}$. Combining this with our earlier result we may conclude that $\mathbf{\tilde{x}}$ consists of the values of the basic interdependent variables and slack variables.
\end{proof}

After obtaining the values of all basic interdependent variables and slack variables, the values of the basic independent variables can be computed in the same way as in standard network simplex \cite[pp.~413--415]{networkflows} within each tree. With this procedure in mind, we are finally ready to describe the change of basis procedure.

\subsubsection{{Change of Basis Procedure}}
\label{subsubsec:changeofbasis}

{In this section we define a procedure for bringing the entering variable into the basis and choosing a variable to leave the basis.} We consider here only entering variables that begin at their lower bound, but an analogous procedure may be described for those at their upper bound. Pivoting consists of increasing the entering variable by $\theta$ units while some of the basic variables change in response in order to maintain feasibility. Assuming finite capacities, at some increase $\theta = \theta^*$ the entering variable or a basic variable will reach one of its bounds, becoming a \textit{blocking variable} (with infinite capacities there may be no blocking variable, in which case the algorithm can terminate with an objective value of $-\infty$). A standard pivoting rule \cite[p.~111]{linopt} can be used to select a blocking variable to leave the basis into $L$ or $U$ while the entering variable takes its place in $B$, at which point the structure of the basic forest $F$ is updated.

We will use $\Delta x_{ij}$ (and $\Delta s_{ij}^{kl}$) to refer to the increase in a given variable as a result of increasing the entering variable by $\theta$ units. Our main goal is to determine the vector of increases $\Delta \mathbf{x}$. The process for doing so depends on whether changing the entering variable causes flow to be transferred between different components of $F$.

\paragraph{Case 1} If increasing the entering variable causes no flow to be transferred between components of $F$, then pivoting can be conducted similarly to standard network simplex \cite[p.~284]{linopt}. This case can occur when the entering variable is an independent arc or an interdependent arc with a basic slack variable, and when both endpoints of the arc lie within the same tree, in which case $\theta$ units of flow are circulated around the unique cycle defined by the arc in the tree (and $\Delta s_{ij}^{kl} = -\hat{a}_{ij} \theta$ for the basic linked slack variable, if it exists). It can also occur when the entering variable is interdependent with a basic linked arc whose endpoints both lie in the same component of $F$, in which case additional flow is circulated around the unique cycle defined by the entering arc's partner. The amount of flow circulated is $-\theta/\hat{a}_{ij}$ if the entering variable is slack, $\alpha_{ij}^{kl} \theta$ if it is the parent of a child $kl$, and $\theta/\alpha_{kl}^{ij}$ if it is the child of a parent $kl$.

\paragraph{Case 2} If increasing the entering variable \textit{does} transfer flow between different components of $F$, we apply a two-step process: first determine the flow increases $\Delta \mathbf{\tilde{x}}$ of the basic interdependent variables, and then determine the flow increases of the basic independent arcs independently within each component of $F$. For the first step, if the entering variable is an arc $ij$ with $i \in T_h$, $j \in T_g$, and $h \ne g$, then increasing its flow effectively increases the net requirements $b(T_h)$ and $b(T_g)$ by $-\theta$ and $\theta$, respectively. If $ij$ is interdependent, then the value $b(t)$ associated with its interdependence is also increased by $-\hat{a}_{ij} \theta$. These changes in net requirements can be accommodated by solving the following perturbed version of the system $\mathbf{D} \mathbf{\tilde{x}} = \mathbf{\tilde{b}}$.
\begin{align}
	\label{eqn:xtildeperturbed} \mathbf{D} \mathbf{\tilde{x}}= \left[ \begin{array}{c c c c c c c|c c c c c} b(T_1) & \hspace{-0.33em} \cdots \hspace{-0.33em} & b(T_h)-\theta & \hspace{-0.33em} \cdots \hspace{-0.33em} & b(T_g)+\theta & \hspace{-0.33em} \cdots \hspace{-0.33em} & b(T_{r-p}) & b(1) & \hspace{-0.33em} \cdots \hspace{-0.33em} & b(t)-\hat{a}_{ij} \theta & \hspace{-0.33em} \cdots \hspace{-0.33em} & b(p) \end{array} \right]'
\end{align}

The instance of $\theta$ in row $t$ of the above system may be ignored if $ij$ is independent. If the entering variable is slack then we may ignore the instances of $\theta$ in rows $T_h$ and $T_g$ and replace $\hat{a}_{ij}$ with $-1$. In any case solving system (\ref{eqn:xtildeperturbed}) gives a vector $\mathbf{\tilde{x}}$ describing the values of the basic interdependent variables after the entering variable has increased by $\theta$ units, from which $\Delta \mathbf{\tilde{x}}$ can be calculated. These are then used in the second step to calculate net requirement increases for all $i \in V$, which can in turn be used to calculate basic independent arc increases using the procedure described in Section \ref{subsubsec:basicvalues} to process the arcs within each component of $F$.\\

In either case, upon completion we obtain a vector $\Delta \mathbf{x}$ of basic (and entering) variable increases. These can be used to calculate the maximum increase $\theta^*$ as well as the blocking variables as in network simplex \cite[p.~290]{linopt}. We then decide on a leaving variable and update the basis structure $(B,L,U)$, the solution vector $\mathbf{x}$, the basic forest $F$, and the matrix $\mathbf{D}$.

\subsection{Computational Complexity}
\label{subsec:complexity}

{As a specialized form of simplex, the modified network simplex algorithm defined in the previous sections} goes through exactly the same number of basis changes as {simplex}: exponentially many in the worst case, but empirically only polynomially many \cite[pp.~127--128]{linopt}. For this reason a more useful point of comparison is to measure the time complexity of a single iteration of both algorithms. Table \ref{table:complexity} shows the time complexities and memory requirements of two {standard implementations of simplex} \cite[p.~107]{linopt} alongside our modified network simplex algorithm, as applied to \mbox{{the \mbox{LIDM-B}}}.

\begin{table}[h]
	\centering
	\begin{tabular}{c c c}
		\hline & \textbf{Memory Requirement} & \textbf{Worst-Case Time} \\
		\hline \textbf{Simplex (Tableau)} & $O((m+p)(n+p))$ & $O((m+p)(n+p))$ \\
		\textbf{Simplex (Revised)} & $O((m+p)^2)$ & $O((m+p)(n+p))$ \\
		\textbf{Modified Network Simplex} & $O(m+n+p^2)$ & $O(m+n+p^3)$ \\
		\hline
	\end{tabular}
	\caption{Memory requirements and worst-case time complexities (in number of elementary operations) for the full tableau implementation of simplex, revised simplex, and our modified network simplex algorithm, applied to the $(m+p) \times (n+p)$ system in \mbox{{the \mbox{LIDM-B}}}.}
	\label{table:complexity}
\end{table}

The memory requirement for our algorithm is $O(m+n+p^2)$, dominated by the $O(m)$ node-level attributes, $O(n)$ arc-level attributes, and $O(p^2)$ elements of $\mathbf{D}$. The time complexity is $O(m+n+p^3)$, dominated by the $O(m)$ operations to set node potentials and net requirements, $O(n)$ operations to set flow values, and $O(p^3)$ operations to solve systems $\mathbf{D}' \boldsymbol{\sigma} = \mathbf{c^{\boldsymbol{\pi}}}$ and $\mathbf{D} \mathbf{\tilde{x}} = \mathbf{\tilde{b}}$ by Gaussian elimination. For the applications of the \mbox{MCNFLI} described earlier in Section \ref{sec:applications} it is reasonable to assume that a relatively small portion of the network's arcs are interdependent. {In particular if $p$ is $O({n}^{1/3})$ then both the tableau and the revised implementations of} simplex have memory and time requirements that are quadratic in $m$ and $n$, while our specialized algorithm's requirements are only linear in $m$ and $n$. As with {network simplex}, this allows it to become significantly more efficient than standard simplex for large networks \cite[p.~287]{linopt}.

\section{Conclusion and Future Work}
\label{sec:conclusion}

The \mbox{MCNFLI} model introduced in Section~\ref{sec:formulation} began as simply an LP relaxation of the binary input dependence model developed by Lee et al. From our discussion we see that the \mbox{MCNFLI} has some interesting modeling applications beyond its original purpose to interdependent networks where the flows through certain arcs are bounded by piecewise linear concave functions of the flows through other arcs. In Section~\ref{sec:applications} we also see how it can be used to obtain approximate solutions and near-optimal feasible mixed integer solutions to the original MILP. {In particular the objective value of the LP relaxation is typically extremely similar to that of the MILP (within $1.7\%$ across all trials), and in the case of \textsc{Structured-Type} trials nearly all randomized rounding schemes were able to find a feasible MILP solution after a single iteration, and typically with an extremely small relative error (less than $6.5\%$ across all trials). Moreover,} due to its special basis structure the modified network simplex algorithm described in Section \ref{sec:simplex} allows for it to be solved much faster than with general LP solution techniques. If the number of interdependencies is sufficiently small in comparison to the size of the overall network we can even achieve a running time that approaches that of network simplex, in spite of the fact that network simplex cannot be directly applied to this problem. {In particular the time requirement of this generalization is $O(m+n+p^3)$, which is equivalent to the $O(m+n)$ complexity of network simplex when $p$ is $O({n}^{1/3})$.}

From the computational results presented in Section \ref{subsec:computational} it is clear that the linear relaxation of the {\mbox{BIDM}} typically produces extremely similar objective values while being significantly less computationally expensive to solve, {though it should be noted that these results are based on a relatively limited range of network topologies, and that more computational testing would be required to extend these results to a broader range of applications. Regardless, the computational savings associated with the linear relaxation could become even more pronounced within a larger model that requires solving instances of {the \mbox{BIDM}} repeatedly during its solution process, for example in the disaster recovery models by Cavdaroglu et al.\ \cite{cavdaroglu2013,cavdaroglu2010,nurre2012,nurre2014,sharkey2015}.} Of particular interest are applications for which only the objective value of {the \mbox{BIDM}} is significant within the overall model and not the flow vector, itself, since in this case the objective of {the \mbox{LIDM}} can be used as a direct substitute for that of {the \mbox{BIDM}} without requiring a randomized rounding algorithm to obtain a feasible solution. Rumpf~2020 \cite{rumpfthesis} explored applications of the \mbox{MCNFLI} model to network interdiction games for planning interdependent network defenses against intelligent attacks, but further computational experiments in this area are needed.

\section*{Acknowledgments}

The authors would like to thank Dr.~Lili Du, Dr.~Robert Ellis, Dr.~Zongzhi Li, and Dr.~Michael Pelsmajer for their valuable input, as well as Susanne Mathies for assistance regarding their previous research. The random problem instance generator mentioned in Section \ref{subsec:computational} is based on a NETGEN implementation programmed in C by Norbert Schlenker. Thank you also to Daniel Simmons, Dr.~Qipeng Phil Zheng, and Dr.~Leandro C.\ Coelho for their CPLEX guides which proved invaluable in setting up our computational trials.
And, thanks to the anonymous referees for their comments that helped improve the exposition of the paper.

\small




\appendix

\section{Simplified Basis Characterization}
\label{app:dhat}

In this appendix we describe an alternative way to characterize the basis that does not require storing the entire matrix $\mathbf{D}$ as defined in (\ref{eqn:ddef}). For the purposes of this explanation we will call a basic interdependent arc \textit{loose} if its linked slack variable is basic, and \textit{tight} otherwise. Suppose that the columns are arranged to that all tight arcs are listed before all loose arcs, and that the rows are arranged so that all interdependencies with nonbasic slack variables are listed above all interdependencies with basic slack variables. Then $\mathbf{Q}$ has the block structure
\begin{align}
	\mathbf{Q} = \begin{bmatrix}
		\mathbf{Q}_t & \\
		& \mathbf{Q}_l
	\end{bmatrix}
\end{align}

where $\mathbf{Q}_t$ contains only the columns and rows corresponding to tight arcs, and $\mathbf{Q}_l$ contains only the columns and rows corresponding to loose arcs. Similarly, each submatrix $\mathbf{B}_h$ has the block structure $\mathbf{B}_h = \begin{bmatrix} \mathbf{B}_t^h & \mathbf{B}_l^h \end{bmatrix}$. Then (\ref{eqn:bstructure}) becomes
\begin{align}
	\mathbf{B} = \left[ \begin{array}{c|c|c|c|c|c|c|c}
		\mathbf{T}_1 & & & & \mathbf{B}_t^1 & \mathbf{B}_l^1 & & \\
		\hline & \mathbf{T}_2 & & & \mathbf{B}_t^2 & \mathbf{B}_l^2 & & \\
		\hline & & \ddots & & \vdots & \vdots & & \\
		\hline & & & \mathbf{T}_{r-p+1} & \mathbf{B}_t^{r-p+1} & \mathbf{B}_l^{r-p+1} & & \\
		\hline & & & & \mathbf{Q}_t & & & \\
		\hline & & & & & \mathbf{Q}_l & \mathbf{I}_1 & \\
		\hline & & & & & & & \mathbf{I}_2
	\end{array} \right]
\end{align}

where $\mathbf{I}_1$ and $\mathbf{I}_2$ represent identity matrices of the appropriate dimensions. Starting from this block structure and going through the same steps as in the proof of Lemma \ref{lemma:characterization} results in the following block structure for $\mathbf{D}$.
\begin{align}
	\label{eqn:ddefblock} \mathbf{D} = \left[ \begin{array}{c c c|c c c|c|c}
		\delta_1^1 & \cdots & \delta_{r_t}^1 & \delta_{r_t+1}^1 & \cdots & \delta_{r_t+r_l}^1 & & \\
		\vdots & \ddots & \vdots & \vdots & \ddots & \vdots & & \\
		\delta_1^{r-p} & \cdots & \delta_{r_t}^{r-p} & \delta_{r_t+1}^{r-p} & \cdots & \delta_{r_t+r_l}^{r-p} & & \\
		\hline & \mathbf{Q}_t & & & & & & \\
		\hline & & & & \mathbf{Q}_l & & \mathbf{I}_1 & \\
		\hline & & & & & & & \mathbf{I}_2
	\end{array} \right]
\end{align}

Here, $r_t$ is the number of tight arcs and $r_l$ is the number of loose arcs. This matrix $\mathbf{D} \in \mathbb{R}^{r \times r}$ is identical to the version shown in (\ref{eqn:ddef}), simply arranged into a different block structure. By Lemma \ref{lemma:characterization}, $\mathbf{B}$ is rank $m+p-1$ if and only if $\mathbf{D}$ is full-rank. Consider the submatrix $\mathbf{\hat{D}}$ of $\mathbf{D}$ defined by
\begin{align}
	\label{eqn:dhatdef} \mathbf{\hat{D}} := \left[ \begin{array}{c c c|c c c}
		\delta_1^1 & \cdots & \delta_{r_t}^1 & \delta_{r_t+1}^1 & \cdots & \delta_{r_t+r_l}^1 \\
		\vdots & \ddots & \vdots & \vdots & \ddots & \vdots \\
		\delta_1^{r-p} & \cdots & \delta_{r_t}^{r-p} & \delta_{r_t+1}^{r-p} & \cdots & \delta_{r_t+r_l}^{r-p} \\
		\hline & \mathbf{Q}_t & & & &
	\end{array} \right]
\end{align}

This matrix $\mathbf{\hat{D}} \in \mathbb{R}^{(r_t+r_l) \times (r_t+r_l)}$ is created by dropping from $\mathbf{D}$ all columns corresponding to basic slack variables and all rows corresponding to interdependencies whose associated slack variable is basic. Note that it is possible to have $r_t = r_l = 0$, in which case $\mathbf{\hat{D}}$ is $0 \times 0$ (this occurs exactly when there are no basic interdependent arcs, or equivalently when all $p$ slack variables are basic). For convenience we consider a $0 \times 0$ matrix to be full-rank. This allows the following observation:

\begin{observation}
	\label{obs:dhat} Matrix $\mathbf{D}$ as defined in (\ref{eqn:ddef}) is full-rank if and only if $\mathbf{\hat{D}}$ as defined in (\ref{eqn:dhatdef}) is full-rank.
\end{observation}

The proof of this relies on the fact that the rows dropped from $\mathbf{D}$ to obtain $\mathbf{\hat{D}}$ each contain the only nonzero entry in one of the columns that is also dropped. Thus the dropped rows are always linearly independent when taken together with the remaining rows, and their contents does not affect the rank of $\mathbf{D}$. Then we have the following corollary:

\begin{corollary}
	\label{cor:dhat} The rank of matrix $\mathbf{B}$ is $m+p-1$ if and only if matrix $\mathbf{\hat{D}}$ as defined in (\ref{eqn:dhatdef}) is full-rank.
\end{corollary}

This follows immediately from Lemma \ref{lemma:characterization} and Observation \ref{obs:dhat}. The advantage of maintaining matrix $\mathbf{\hat{D}}$ instead of $\mathbf{D}$ is that $\mathbf{\hat{D}}$ is smaller (its dimensions are between $0 \times 0$ and $2p \times 2p$, in contrast to the dimensions of $\mathbf{D}$ which are between $p \times p$ and $3p \times 3p$), and can thus be used to reduce the sizes of the linear systems that need to be solved during the modified network simplex algorithm presented in Section \ref{sec:simplex}. This requires some minor modifications to the procedures defined above.

Within the reduced cost calculation algorithm described in Section \ref{subsec:reducedcosts}, the system $\mathbf{D}' \boldsymbol{\sigma} = \mathbf{c^{\boldsymbol{\pi}}}$ is solved to obtain a potential correction term for each component of $F$ and each slack variable. Let $\mathbf{\hat{\boldsymbol{\sigma}}}$ and $\mathbf{\hat{c}^{\boldsymbol{\pi}}}$ be the subvectors of $\boldsymbol{\sigma}$ and $\mathbf{c^{\boldsymbol{\pi}}}$, respectively, which exclude the elements corresponding to basic slack variables. The proof of Proposition \ref{prop:dsigma} shows that the system $\mathbf{D}' \boldsymbol{\sigma} = \mathbf{c^{\boldsymbol{\pi}}}$ results in $\sigma_{ij}^{kl} = 0$ for all basic slack variables $s_{ij}^{kl}$, and so it suffices to solve the system $\mathbf{\hat{D}}' \mathbf{\hat{\boldsymbol{\sigma}}} = \mathbf{\hat{c}^{\boldsymbol{\pi}}}$ in place of $\mathbf{D}' \boldsymbol{\sigma} = \mathbf{c^{\boldsymbol{\pi}}}$ and then set $\sigma_{ij}^{kl} \equiv 0$ for all basic slack variables.

Within Case 2 of the change of basis algorithm described in Section \ref{subsubsec:changeofbasis}, the perturbed system $\mathbf{D} \mathbf{\tilde{x}} = \mathbf{\tilde{b}}$ is solved to determine the values of the basic interdependent variables. Let $\mathbf{\hat{x}}$ be the subvector of $\mathbf{\tilde{x}}$ which excludes elements corresponding to basic slack variables, and let $\mathbf{\hat{b}}$ be the subvector of $\mathbf{\tilde{b}}$ which excludes elements corresponding to the interdependencies of loose arcs. Then the system {$\mathbf{\hat{D}} \mathbf{\hat{x}} = \mathbf{\hat{b}}$} can be solved to obtain the flows $\mathbf{\hat{x}}$ of all basic interdependent arcs, which along with the definitions of $L$ and $U$ gives the flows of all interdependent arcs. These flows, along with the linking equations (\ref{eqn:mcnflilinking2}), give the values of all basic slack variables. Because $\mathbf{\hat{x}}$ excludes basic slack variables, additional steps are required to calculate the entries of $\Delta \mathbf{\tilde{x}}$ corresponding to slack variables. If the incoming variable has a basic linked slack variable, say $s_{ij}^{kl}$, then $\Delta s_{ij}^{kl} = \hat{a}_{ij} \theta$. If $s_{ij}^{kl}$ is basic and both linked arcs are also basic, then $\Delta s_{ij}^{kl} = \alpha_{ij}^{kl} \Delta x_{ij} - \Delta x_{kl}$ (if $ij$ is nonbasic then we can set $\Delta x_{ij} \equiv 0$, and likewise if $kl$ is nonbasic we can set $\Delta x_{kl} \equiv 0$).

Applying these improvements which use $\mathbf{\hat{D}}$ in place of $\mathbf{D}$ does not change the worst-case time complexity of the modified network simplex algorithm, which remains $O(m+n+p^3)$ since in the worst case $\mathbf{\hat{D}}$ still has dimensions of $O(p)$. However in practice $\mathbf{\hat{D}}$ is much smaller than $\mathbf{D}$, resulting in a practical reduction in computational time as well as improving the best-case time complexity from $O(n+m+p^3)$ to $O(n+m)$, which occurs when $\mathbf{\hat{D}}$ is $0 \times 0$.

\section{Examples of Randomized Rounding Failure}
\label{app:rrfailure}

The {\mbox{$\textsc{RR-Child}(0.00)$}} and {\mbox{$\textsc{RR-Parent}(0.00)$}} randomized rounding schemes described in Section \ref{subsubsec:rr} carry the potential of failing to terminate for certain problem instances. This can occur if $P_{ij}^{kl}$ takes a value of either 0 or 1, in which case the scheme will choose the same value for $y_{ij}^{kl}$ in every iteration, even if that decision can never lead to a feasible MILP solution. Figures \ref{subfig:rrfail1} and \ref{subfig:rrfail2} show two different ways in which this can occur.

Figure \ref{subfig:rrfail1} shows an example in which {\mbox{$\textsc{RR-Child}(0.00)$}} and {\mbox{$\textsc{RR-Parent}(0.00)$}} both attempt in each iteration to force the saturation of an arc which is unusable in any feasible MILP solution. In this network the optimal LP solution saturates both $(7,8)$ and $(9,10)$, which form a parent/child pair, and so {\mbox{$\textsc{RR-Child}(0.00)$}} and {\mbox{$\textsc{RR-Parent}(0.00)$}} both set $P_{ij}^{kl}$ equal to 1. However, child arcs $(5,7)$ and $(9,10)$ cannot be used in any feasible MILP solution as their parent arcs $(2,3)$ and $(7,8)$ cannot be saturated due to bottlenecks. As a result, {\mbox{$\textsc{RR-Child}(0.00)$}} and {\mbox{$\textsc{RR-Parent}(0.00)$}} both create an infeasible MILP in each iteration.

Similarly, Figure \ref{subfig:rrfail2} shows an example in which {\mbox{$\textsc{RR-Child}(0.00)$}} and {\mbox{$\textsc{RR-Parent}(0.00)$}} both attempt in each iteration to disallow the use of an arc which is required in any feasible MILP solution. In this network the optimal LP solution sends no flow through either $(6,7)$ or $(8,9)$, which form a parent/child pair, resulting in both {\mbox{$\textsc{RR-Child}(0.00)$}} and {\mbox{$\textsc{RR-Parent}(0.00)$}} setting $P_{ij}^{kl}$ equal to 0. However in this case all feasible MILP solutions require the use of both $(6,7)$ and $(8,9)$, and so again each iteration of {\mbox{$\textsc{RR-Child}(0.00)$}} and {\mbox{$\textsc{RR-Parent}(0.00)$}} result in an infeasible MILP.

\FloatBarrier

\begin{figure}[p]
	\centering
	\subcaptionbox{Failure due to LP solution saturating arcs that no feasible MILP solution can use. Interdependencies: $x_{(5,7)} \le x_{(2,3)}$, $x_{(9,10)} \le x_{(7,8)}$\label{subfig:rrfail1}}{ {\includegraphics[width=0.4\textwidth]{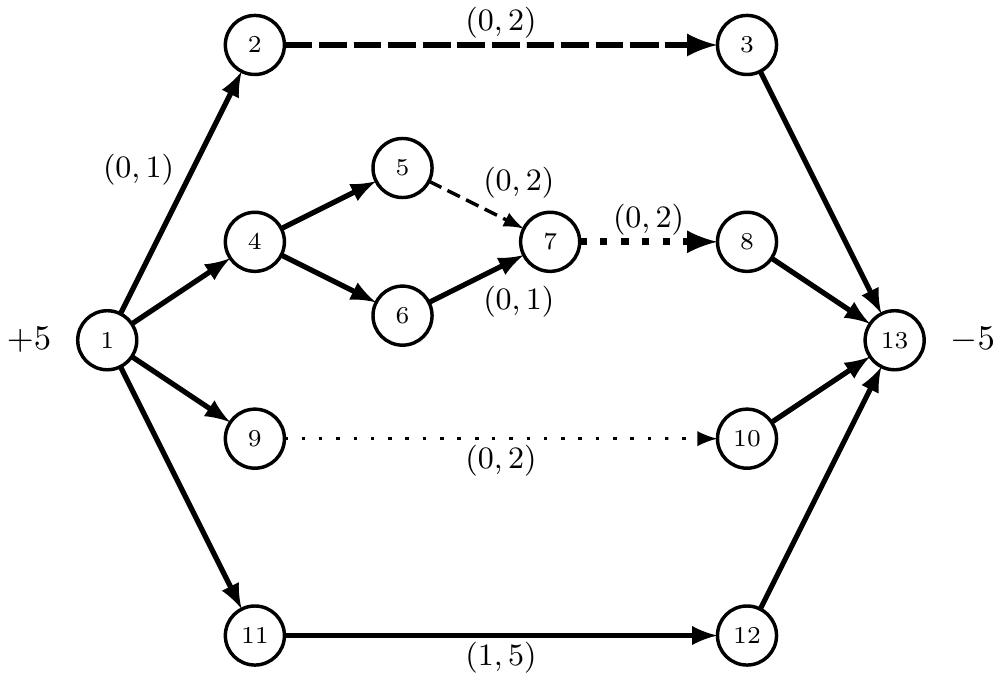}}} \hspace{0.5in}
	\subcaptionbox{Failure due to LP solution excluding arcs required by any feasible MILP solution. Interdependencies: $x_{(4,5)} \le 2 x_{(2,3)}$, $x_{(8,9)} \le x_{(6,7)}$\label{subfig:rrfail2}}{ {\includegraphics[width=0.4\textwidth]{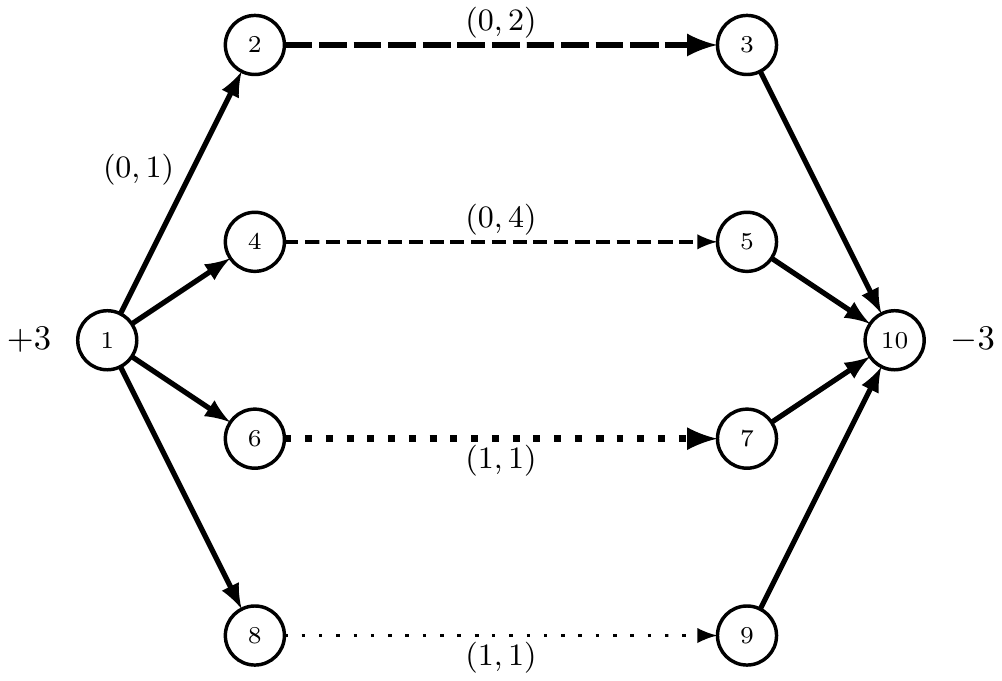}}}
	\caption{Example networks for which {\mbox{$\textsc{RR-Child}(0.00)$}} and {\mbox{$\textsc{RR-Parent}(0.00)$}} both fail in every iteration. The number next to each node indicates the supply value at that node, while the ordered pair next to each arc indicates its cost and capacity, respectively. Unlisted supply values are zero, while unlisted cost/capacity pairs are $(0,\infty)$. Dashed and dotted arcs indicate interdependent pairs, with the bolder arc indicating the parent and the thinner arc indicating the child in each pair. Interdependencies are listed under each network.}
	\label{fig:rrfail}
\end{figure}

\begin{figure}[p]
	\centering
	\subcaptionbox{Example for which the LP relaxation of {the \mbox{BIDM}} produces an arbitrarily bad lower bound for the optimal cost.\label{subfig:badexample}}{ {\includegraphics[width=0.3\textwidth]{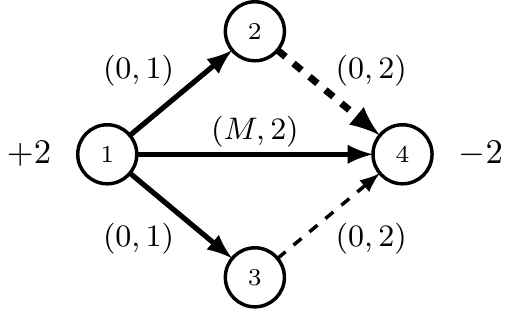}}} \hspace{0.5in}
	\subcaptionbox{Example for which the randomized rounding solutions of {the \mbox{BIDM}} can produce arbitrarily bad upper bounds for the optimal cost.\label{subfig:badexamplerr}}{ {\includegraphics[width=0.3\textwidth]{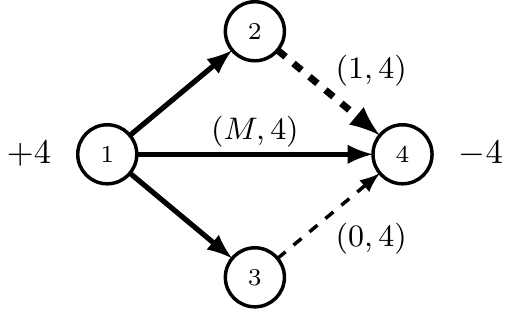}}}
	\caption{Example networks showing that the LP relaxation of {the \mbox{BIDM}} and the randomized rounding solutions can produce arbitrarily bad upper and lower bounds for the optimal cost. The labeling system is the same as used in Figure \ref{fig:rrfail}. The arc cost $M$ is an arbitrary constant.}
	\label{fig:badexample}
\end{figure}

\FloatBarrier

\section{Examples of Arbitrarily Bad Theoretical Bounds}
\label{app:boundexample}

As noted in Section \ref{subsubsec:rr}, the LP relaxation of {the \mbox{BIDM}} provides a lower bound for its optimal value while any randomized rounding solution provides an upper bound, but both of these bounds may be arbitrarily far away from the true value. Figures \ref{subfig:badexample} and \ref{subfig:badexamplerr} show examples for which each of these is the case.

In the network shown in Figure \ref{subfig:badexample}, the optimal LP solution sends 1 unit of flow through node 2 to half-saturate parent arc $(2,4)$, allowing the remaining 1 unit of flow to be sent through node 3 to take advantage of the half-usable capacity of child arc $(3,4)$, at a flow cost of 0. The only feasible MILP solution sends all flow through arc $(1,4)$ at a cost of $M$, and since $M$ can be arbitrarily large, the LP objective can be arbitrarily far below the MILP objective.

Similarly, in the network shown in Figure \ref{subfig:badexamplerr} the optimal LP solution splits the flow equally between the path through node 2 and the path through node 4 by half-saturating the parent/child pair $(2,4)$ and $(3,4)$. Assuming $M > 1$, the optimal MILP solution sends all flow through node 2 at a total cost of $4$. Because both the parent and the child are half-saturated in the LP solution, any randomized rounding scheme has a probability of 0.5 of either forcing the use of $(2,4)$ or disallowing the use of $(3,4)$. If $(3,4)$ is disallowed, then the only remaining feasible MILP solution sends all flow through $(1,4)$ at a cost of $4M$, and so the randomized rounding objective can be arbitrarily far above the MILP objective.

\section{Numerical Example}
\label{app:byhand}

In order to illustrate how our proposed solution algorithm works in practice we will conduct a few iterations on the example network shown in Figure \ref{fig:exbase}. This network $G$ is based on the example network from Calvete 2003 \cite{calvete2003}. It has $m=11$ nodes, $n=24$ arcs, and $p=4$ interdependencies, meaning that its basis must contain $m+p-1 = 14$ variables at all times.

\begin{figure}[p]
	\centering
	 {\includegraphics[width=0.3\textwidth]{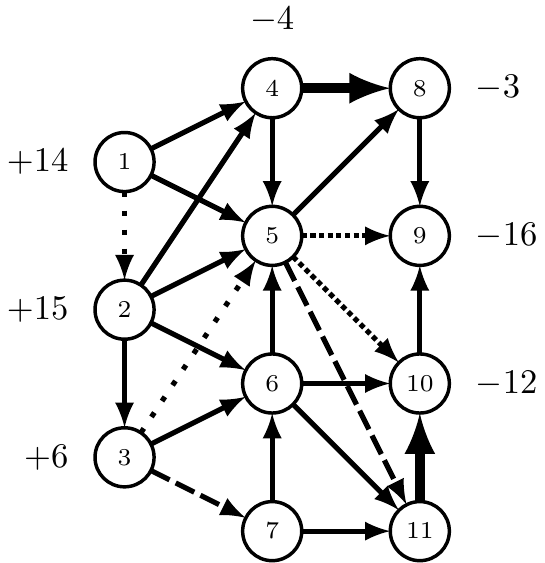}}
	\begin{align*}
		\begin{array}{c c c c c c c c}
			c_{(1,2)} = 5 & c_{(1,4)} = 8 & c_{(1,5)} = 12 & c_{(2,3)} = \frac{1}{2} & c_{(2,4)} = \frac{1}{2} & c_{(2,5)} = \frac{1}{2} & c_{(2,6)} = 1 & c_{(3,5)} = 10 \\
			c_{(3,6)} = \frac{1}{2} & c_{(3,7)} = 3 & c_{(4,5)} = 8 & c_{(4,8)} = 5 & c_{(5,8)} = 1 & c_{(5,9)} = 1 & c_{(5,10)} = 1 & c_{(5,11)} = 3 \\
			c_{(6,5)} = 3 & c_{(6,10)} = \frac{1}{2} & c_{(6,11)} = \frac{1}{2} & c_{(7,6)} = 5 & c_{(7,11)} = 4 & c_{(8,9)} = 4 & c_{(10,9)} = 2 & c_{(11,10)} = 6
		\end{array}
	\end{align*}
	\begin{align*}
		\begin{array}{c c c c}
			x_{(11,10)} \le \frac{1}{2} x_{(4,8)} + 1, & x_{(5,11)} \le \frac{1}{2} x_{(3,7)} + 2, & x_{(1,2)} \le x_{(3,5)} + \frac{1}{2}, & x_{(5,10)} \le \frac{1}{2} x_{(5,9)}
		\end{array}
	\end{align*}
	\caption{Example network $G$ on which to conduct the modified network simplex algorithm. All arc capacities are 15. Supply values are listed next to each node (no value listed means zero). Arc costs are listed below the network. Interdependent arcs are highlighted, with the specific interdependencies listed below the network.}
	\label{fig:exbase}
\end{figure}

\begin{figure}[p]
	\centering \quad
	\subcaptionbox{Initial basis.\label{subfig:ex2}}{ {\includegraphics[width=0.2\textwidth]{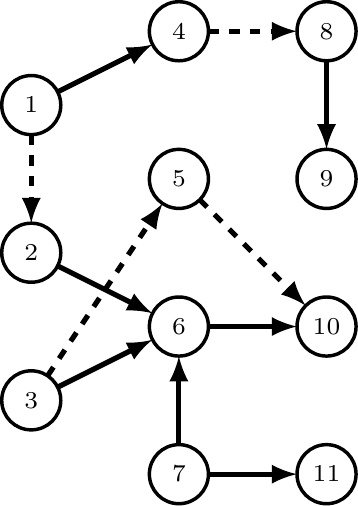}}} \hspace{0.5in}
	\subcaptionbox{Iteration 1.\label{subfig:ex3}}{ {\includegraphics[width=0.2\textwidth]{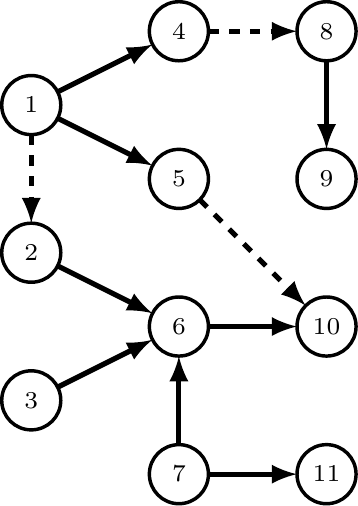}}}
	\hspace{0.5in}
	\subcaptionbox{Iteration 2.\label{subfig:ex4}}{ {\includegraphics[width=0.2\textwidth]{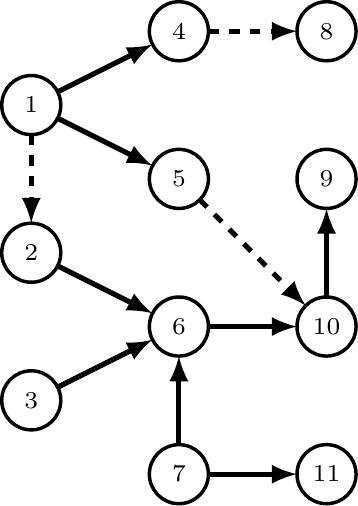}}}
	\caption{Basis structure of the example network after each iteration. Only the basic arcs are shown. Solid lines represent basic independent arcs, while dashed lines represent basic interdependent arcs. All bases contain slack variables $s_{(4,8)}^{(11,10)}$, $s_{(3,7)}^{(5,11)}$, and $s_{(5,9)}^{(5,10)}$.}
	\label{fig:examples}
\end{figure}

\paragraph{Initial Basis} The initial basis contains the arcs shown in Figure \ref{subfig:ex2} as well as slack variables $s_{(4,8)}^{(11,10)}$, $s_{(3,7)}^{(5,11)}$, and $s_{(5,9)}^{(5,10)}$. Most nonbasic variables are in $L$ and currently have a value of zero, while $x_{(2,5)}$ and $x_{(5,9)}$ are in $U$, putting them at their upper bound of 15. The initial values of the basic variables are
\begin{align*}
	\begin{array}{c c c c c c c}
		x_{(1,2)} = 6 & x_{(1,4)} = 8 & x_{(2,6)} = 6 & x_{(3,5)} = \frac{11}{2} & x_{(3,6)} = \frac{1}{2} & x_{(4,8)} = 4 & x_{(5,10)} = \frac{11}{2} \\
		x_{(6,10)} = \frac{13}{2} & x_{(7,6)} = 0 & x_{(7,11)} = 0 & x_{(8,9)} = 1 & s_{(4,8)}^{(11,10)} = 3 & s_{(3,7)}^{(5,11)} = 2 & s_{(5,9)}^{(5,10)} = 2
	\end{array}
\end{align*}

The basic forest $F$ contains four components: $T_1 = \{1,4\}$, $T_2 = \{2,3,6,7,10,11\}$, $T_3 = \{5\}$, and $T_4 = \{8,9\}$. All components, combined, contain the 7 basic independent arcs, while the remaining 7 basic variables are interdependent: $x_{(4,8)}$, $x_{(3,5)}$, $x_{(1,2)}$, $x_{(5,10)}$, $s_{(4,8)}^{(11,10)}$, $s_{(3,7)}^{(5,11)}$, and $s_{(5,9)}^{(5,10)}$. This means that $r=7$, thus $\mathbf{D}$ is $7 \times 7$. Arranging the basic interdependent variables and trees in this order, the matrix $\mathbf{D}$ as described in (\ref{eqn:ddef}) is
\begin{align*}
	\mathbf{D} = \left[
	\begin{array}{r r r r|c c c}
		1 & 0 & 1 & 0 & 0 & 0 & 0 \\
		0 & 1 & -1 & -1 & 0 & 0 & 0 \\
		0 & -1 & 0 & 1 & 0 & 0 & 0 \\
		\hline -\frac{1}{2} & 0 & 0 & 0 & 1 & 0 & 0 \\
		0 & 0 & 0 & 0 & 0 & 1 & 0 \\
		0 & -1 & 1 & 0 & 0 & 0 & 0 \\
		0 & 0 & 0 & 1 & 0 & 0 & 1
	\end{array}
	\right]
\end{align*}

\paragraph{Iteration 1}

To calculate potentials using the method described in Section \ref{subsec:reducedcosts}, we begin by tentatively assigning a potential of zero to each interdependence. For the node potentials we select an arbitrary root for each component of the basic forest and assign it a potential of zero, in this case selecting 1 as the root of $T_1$, 2 as the root of $T_2$, 5 as the root of $T_3$, and 8 as the root of $T_4$. We then use equations (\ref{eqn:rcindependent}) to calculate the remaining node potentials from root to leaf, obtaining a complete potential vector of
\begin{align*}
	\boldsymbol{\pi} = \left[
	\begin{array}{c c c c c c c c c c c|c c c c}
		0 & 0 & -\frac{1}{2} & -8 & 0 & -1 & 4 & 0 & -4 & -\frac{3}{2} & 0 & 0 & 0 & 0 & 0
	\end{array}
	\right]'
\end{align*}

We then use these values in equations (\ref{eqn:rcparent})--(\ref{eqn:rcslack}) for the basic interdependent variables, obtaining
\begin{align*}
	\begin{array}{c c c c c}
		c_{(4,8)}^\pi = 13 & c_{(3,5)}^\pi = \frac{21}{2} & c_{(1,2)}^\pi = 5 & c_{(5,10)}^\pi = -\frac{1}{2} & (c^\pi)_{(4,8)}^{(11,10)} = 0
	\end{array}
\end{align*}

Some of these are nonzero, and so we must calculate a correction term for each forest component and interdependence. Solving $\mathbf{D}' \boldsymbol{\sigma} = \mathbf{c^{\boldsymbol{\pi}}}$ gives a correction vector of
\begin{align*}
	\boldsymbol{\sigma} = \left[
	\begin{array}{c c c|c c c c}
		13 & -2 & -\frac{5}{2} & 0 & 0 & -10 & 0
	\end{array}
	\right]'
\end{align*}

Applying these to the initial potentials $\boldsymbol{\pi}$ using (\ref{eqn:corrected1})--(\ref{eqn:corrected3}) gives a corrected potential vector of
\begin{align*}
	\mathbf{\tilde{\boldsymbol{\pi}}} = \left[
	\begin{array}{c c c c c c c c c c c|c c c c}
		13 & -2 & -\frac{5}{2} & 5 & -\frac{5}{2} & -3 & 2 & 0 & -4 & -\frac{7}{2} & -2 & 0 & 0 & -10 & 0
	\end{array}
	\right]'
\end{align*}

Equipped with the potential values, we can use equations (\ref{eqn:rcindependent})--(\ref{eqn:rcslack}) to calculate the reduced cost of any nonbasic variable. In this case, we might notice that $c_{(1,5)}^\pi = c_{(1,5)} - (\tilde{\pi}_1 - \tilde{\pi}_5) = 12 - (13 + \frac{5}{2}) = -\frac{7}{2}$. Since $x_{(1,5)} \in L$ and it has a negative reduced cost it is a candidate to enter the basis, and so we proceed with a change of basis.

$(1,5)$ is an independent arc and bridges the gap between $T_1$ and $T_3$, making this a Case 2 change of basis.  We first determine the net requirement values $\mathbf{\tilde{b}}$ as defined in Section \ref{subsubsec:basicvalues}. For each component of $F$ we find the total supply value and inflow from arcs in $U$, resulting in $b(T_1) = 10$, $b(T_2) = -6$, and $b(T_3) = 0$. Interdependencies $t=1,2,3$ do not contain any variables in $U$, and so by (\ref{eqn:btdef}) we have $b(1) = \beta_{(4,8)}^{(11,10)} = 1$, $b(2) = \beta_{(3,7)}^{(5,11)} = 2$, and $b(3) = \beta_{(3,5)}^{(1,2)} = \frac{1}{2}$. Interdependent arc $(5,9)$ is in $U$, so $b(4) = \beta_{(5,9)}^{(5,10)} + \alpha_{(5,9)}^{(5,10)} u_{(5,9)} = \frac{15}{2}$.

We use these values to form the perturbed system $\mathbf{D} \mathbf{\tilde{x}} = \mathbf{\tilde{b}}$, adding $\theta$ to $b(T_3)$ and $-\theta$ to $b(T_1)$ because the incoming arc leaves $T_1$ and enters $T_3$. The solution to the system is
\begin{align*}
	\mathbf{\tilde{x}} = \left[
	\begin{array}{c c c c|c c c}
		4 & \frac{11}{2} - \theta & 6 - \theta & \frac{11}{2} & 3 & 2 & 2
	\end{array}
	\right]'
\end{align*}

Subtracting the initial basic feasible solution from this vector gives us the change vector
\begin{align*}
	\Delta \mathbf{\tilde{x}} = \left[
	\begin{array}{c c c c|c c c}
		0 & -\theta & -\theta & 0 & 0 & 0 & 0
	\end{array}
	\right]'
\end{align*}

Next we must determine which basic independent variables change. We may ignore $T_3$ because it contains no arcs, and $T_4$ because it is not incident to any of the affected arcs. Within $T_1$ and $T_2$, we must calculate changes in the modified net requirements $\mathring{b}(i)$ of each affected node $i$ as referenced in Section \ref{subsubsec:basicvalues}. We need only calculate the modified requirement values of nodes 1, 2, 3, and 5, since these are the only nodes incident to affected arcs. This gives
\begin{align*}
	\begin{array}{c c c c}
		\Delta \mathring{b}(1) = 0 & \Delta \mathring{b}(2) = \theta & \Delta \mathring{b}(3) = -\theta & \Delta \mathring{b}(5) = 0
	\end{array}
\end{align*}

The only nonzero changes occur at nodes 2 and 3. Finally we can scan through each tree from leaf to root, calculating the change in each independent flow variable as described in Section \ref{subsubsec:basicvalues}. There are only two nonzero changes: $\Delta x_{(2,6)} = -\theta$ and $\Delta x_{(3,6)} = \theta$.

Having determined the full change vector $\Delta \mathbf{x}$, we use the method outlined in Section \ref{subsubsec:changeofbasis} to calculate $\theta^*$ and the blocking variables. The first variable to reach a bound as $\theta$ increases is $x_{(3,5)}$, which reaches zero when $\theta = \frac{11}{2}$. Then the change of basis consists of $(1,5)$ entering the basis from $L$, $(3,5)$ leaving the basis into $L$, and a change increment of $\theta^* = \frac{11}{2}$. The new basis is shown in Figure \ref{subfig:ex3}, and the values of the basic variables are
\begin{align*}
	\begin{array}{c c c c c c c}
		x_{(1,2)} = \frac{1}{2} & x_{(1,4)} = 8 & x_{(2,6)} = \frac{1}{2} & x_{(1,5)} = \frac{11}{2} & x_{(3,6)} = 6 & x_{(4,8)} = 4 & x_{(5,10)} = \frac{11}{2} \\
		x_{(6,10)} = \frac{13}{2} & x_{(7,6)} = 0 & x_{(7,11)} = 0 & x_{(8,9)} = 1 & s_{(4,8)}^{(11,10)} = 3 & s_{(3,7)}^{(5,11)} = 2 & s_{(5,9)}^{(5,10)} = 2
	\end{array}
\end{align*}

\paragraph{Iteration 2}

Due to the previous basis change, components $T_1$ and $T_3$ merge into a single component. We will label the new components $T_1 = \{1,4,5\}$, $T_2 = \{2,3,6,7,10,11\}$, and $T_3 = \{8,9\}$. Under this naming scheme, and keeping the interdependencies in their previous order, we now have
\begin{align*}
	\mathbf{D} = \left[
	\begin{array}{r r r|r r r}
		1 & 1 & 1 & 0 & 0 & 0\\
		0 & -1 & -1 & 0 & 0 & 0 \\
		\hline -\frac{1}{2} & 0 & 0 & 1 & 0 & 0 \\
		0 & 0 & 0 & 0 & 1 & 0 \\
		0 & 1 & 0 & 0 & 0 & 0 \\
		0 & 0 & 1 & 0 & 0 & 1
	\end{array}
	\right]
\end{align*}

Going through the same steps as in the previous iteration, the corrected potential vector is
\begin{align*}
	\mathbf{\tilde{\boldsymbol{\pi}}} = \left[
	\begin{array}{c c c c c c c c c c c|c c c c}
		13 & \frac{3}{2} & 1 & 5 & 1 & \frac{1}{2} & \frac{11}{2} & 0 & -4 & 0 & \frac{3}{2} & 0 & 0 & -\frac{13}{2} & 0
	\end{array}
	\right]'
\end{align*}

From this we can calculate $c_{(10,9)}^\pi -2$, and since $x_{(10,9)} \in L$, it is a candidate to enter the basis. $(10,9)$ is an independent arc that bridges the gap between $T_2$ and $T_3$. This leads to another instance of change of basis Case 2. The perturbed net requirement vector is
\begin{align*}
	\mathbf{\tilde{b}} = \left[
	\begin{array}{c c|c c c c}
		10 & -6-\theta & 1 & 2 & \frac{1}{2} & \frac{15}{2}
	\end{array}
	\right]'
\end{align*}

Solving $\mathbf{D} \mathbf{\tilde{x}} = \mathbf{\tilde{b}}$ gives
\begin{align*}
	\mathbf{\tilde{x}} = \left[
	\begin{array}{c c c|c c c}
		4-\theta & \frac{1}{2} & \frac{11}{2} + \theta & 3 - \frac{1}{2} \theta & 2 & 2-\theta
	\end{array}
	\right]'
\end{align*}

The relevant supply value changes are
\begin{align*}
	\begin{array}{c c c c c}
		\Delta \mathring{b}(4) = -\theta & \Delta \mathring{b}(5) = \theta & \Delta \mathring{b}(8) = \theta & \Delta \mathring{b}(9) = -\theta & \Delta \mathring{b}(10) = 0
	\end{array}
\end{align*}

The only nonzero changes occur for nodes in components $T_1$ and $T_3$, and so only arcs in these trees need be considered. Calculating the changes in their arcs as before, we find three nonzero values: $\Delta x_{(1,4)} = -\theta$, $\Delta x_{(1,5)} = \theta$, and $\Delta x_{(8,9)} = -\theta$.

A total of eight variables changes as $\theta$ increases. The first to reach a bound is $x_{(8,9)}$, which becomes zero when $\theta^* = 1$. Then $x_{(10,9)}$ moves from $L$ into $B$, $x_{(8,9)}$ moves from $B$ into $L$, and the remaining variables are adjusted by substituting $\theta = 1$ into their change terms. The new basis is shown in Figure \ref{subfig:ex4}, and the values of the basic variables are now
\begin{align*}
	\begin{array}{c c c c c c c}
		x_{(1,2)} = \frac{1}{2} & x_{(1,4)} = 7 & x_{(2,6)} = \frac{1}{2} & x_{(1,5)} = \frac{13}{2} & x_{(3,6)} = 6 & x_{(4,8)} = 3 & x_{(5,10)} = \frac{13}{2} \\
		x_{(6,10)} = \frac{13}{2} & x_{(7,6)} = 0 & x_{(7,11)} = 0 & x_{(10,9)} = 1 & s_{(4,8)}^{(11,10)} = \frac{5}{2} & s_{(3,7)}^{(5,11)} = 2 & s_{(5,9)}^{(5,10)} = 1
	\end{array}
\end{align*}

\paragraph{Iteration 3}

If the new components are labeled as $T_1 = \{1,4,5\}$, $T_2 = \{2,3,6,7,9,10,11\}$, and $T_3 = \{8\}$, then $\mathbf{D}$ remains exactly the same as in the previous iteration since all basic independent arcs still bridge the same tree indices. Applying the same potential calculation technique, the corrected potential vector is
\begin{align*}
	\mathbf{\tilde{\boldsymbol{\pi}}} = \left[
	\begin{array}{c c c c c c c c c c c|c c c c}
		13 & \frac{3}{2} & 1 & 5 & 1 & \frac{1}{2} & \frac{11}{2} & 0 & -2 & 0 & \frac{3}{2} & 0 & 0 & -\frac{19}{2} & 0
	\end{array}
	\right]'
\end{align*}

Using these potentials to calculate reduced costs, we find that no variable in $L$ has a negative reduced cost and no variable in $U$ has a positive reduced cost. This implies that the current solution is optimal. We may terminate the algorithm and output the current solution vector $\mathbf{x}$, which has a total cost of $189.25$.

\section{Computational Trial Data Tables}
\label{app:tables}

This section contains the full data tables for all trials described in Section \ref{subsec:computational}. {The raw data summarized in these tables can be viewed online \cite{p1data}}. Tables \ref{table:lpgapnode} and \ref{table:lpgaparc} show the results of the LP relaxation trials for {\textsc{Structured-Type}} and {\textsc{Unstructured-Type}} problems, respectively. Within each table, columns correspond to different percentages of sink nodes (for {\textsc{Structured-Type}}) or arcs (for {\textsc{Unstructured-Type}}) acting as parents, while rows correspond to different network densities. Each table entry gives the mean and standard deviation in the relative error for all successful trials of the given network type, as well as the number ($n$) of trials on which the statistics are based. All tables throughout this section are organized in the same format.

Tables \ref{table:rrc0gapnode}--\ref{table:rrfgapnode} show the results of the {\mbox{$\textsc{RR-Child}(0.00)$}}, {\mbox{$\textsc{RR-Child}(0.01)$}}, {\mbox{$\textsc{RR-Child}(0.05)$}}, and {\mbox{$\textsc{RR-Fair}$}} trials for {\textsc{Structured-Type}} problems, respectively, and corresponds to Figure \ref{fig:rrgapnode}. All results displayed are based on the full set of 60 trials, except for the single failed case of {\mbox{$\textsc{RR-Child}(0.00)$}} for 512 nodes and 10\% of sinks interdependent. Tables \ref{table:rrc0fail}--\ref{table:rrffail} show the failure rates of the {\mbox{$\textsc{RR-Child}(0.00)$}}, {\mbox{$\textsc{RR-Child}(0.01)$}}, {\mbox{$\textsc{RR-Child}(0.05)$}}, and {\mbox{$\textsc{RR-Fair}$}} trials for {\textsc{Unstructured-Type}} problems, respectively, in terms of the percentage of the 60 trials in each category that failed to reach a feasible solution within 1000 attempts. Tables \ref{table:rrc0gaparc}--\ref{table:rrfgaparc} show the relative errors for the {\mbox{$\textsc{RR-Child}(0.00)$}}, {\mbox{$\textsc{RR-Child}(0.01)$}}, {\mbox{$\textsc{RR-Child}(0.05)$}}, and {\mbox{$\textsc{RR-Fair}$}} trials for {\textsc{Unstructured-Type}} problems, respectively, limited to only the successful trials. Tables \ref{table:rrc0fail}--\ref{table:rrffail} correspond to Figure \ref{subfig:rrtarc} while Tables \ref{table:rrc0gaparc}--\ref{table:rrfgaparc} correspond to Figure \ref{subfig:rrgaparc}.

\begin{table}[p]
	\centering
	\begin{tabular}{l l l r r r r}
		\hline \multicolumn{7}{l}{\textbf{{\textsc{Structured-Type}}, LP}} \\
		Nodes & Arcs/Node & & \multicolumn{1}{l}{2\%} & \multicolumn{1}{l}{5\%} & \multicolumn{1}{l}{10\%} & \multicolumn{1}{l}{15\%} \\
		\hline 256 & 4 & Mean & 0.03218\% & 0.09867\% & 0.08198\% & 0.06282\% \\
		& & Std. Dev. & 0.09196\% & 0.29471\% & 0.20290\% & 0.12985\% \\
		& & $n$ & \multicolumn{1}{l}{60} & \multicolumn{1}{l}{60} & \multicolumn{1}{l}{60} & \multicolumn{1}{l}{60} \\
		\hhline{~------} & 8 & Mean & 0.02036\% & 0.03139\% & 0.06379\% & 0.05054\% \\
		& & Std. Dev. & 0.06776\% & 0.13905\% & 0.24593\% & 0.14242\% \\
		& & $n$ & \multicolumn{1}{l}{60} & \multicolumn{1}{l}{60} & \multicolumn{1}{l}{60} & \multicolumn{1}{l}{60} \\
		\hline 512 & 4 & Mean & 0.01396\% & 0.06195\% & 0.07669\% & 0.13231\% \\
		& & Std. Dev. & 0.03793\% & 0.11006\% & 0.10339\% & 0.18200\% \\
		& & $n$ & \multicolumn{1}{l}{60} & \multicolumn{1}{l}{60} & \multicolumn{1}{l}{60} & \multicolumn{1}{l}{60} \\
		\hhline{~------} & 8 & Mean & 0.00692\% & 0.00984\% & 0.03386\% & 0.02207\% \\
		& & Std. Dev. & 0.02401\% & 0.03239\% & 0.06100\% & 0.06443\% \\
		& & $n$ & \multicolumn{1}{l}{60} & \multicolumn{1}{l}{60} & \multicolumn{1}{l}{60} & \multicolumn{1}{l}{60} \\
		\hline
	\end{tabular}
	\caption{Relative error for LP relaxation in {\textsc{Structured-Type}} trials. Columns indicate percentage of sink nodes acting as parents.}
	\label{table:lpgapnode}
\end{table}

\begin{table}[p]
	\centering
	\begin{tabular}{l l l r r r r}
		\hline \multicolumn{7}{l}{\textbf{{\textsc{Unstructured-Type}}, LP}} \\
		Nodes & Arcs/Node & & \multicolumn{1}{l}{1\%} & \multicolumn{1}{l}{2\%} & \multicolumn{1}{l}{5\%} & \multicolumn{1}{l}{10\%} \\
		\hline 256 & 4 & Mean & {0.12551\%} & 0.33397\% & {0.40693\%} & 1.60084\% \\
		& & Std. Dev. & {0.28373\%} & 0.46781\% & {0.57130\%} & 1.48633\% \\
		& & $n$ & \multicolumn{1}{l}{60} & \multicolumn{1}{l}{60} & \multicolumn{1}{l}{60} & \multicolumn{1}{l}{60} \\
		\hhline{~------} & 8 & Mean & {0.08949\%} & 0.13669\% & {0.28389\%} & 0.71868\% \\
		& & Std. Dev. & {0.29966\%} & 0.31658\% & {0.46065\%} & 1.06470\% \\
		& & $n$ & \multicolumn{1}{l}{60} & \multicolumn{1}{l}{60} & \multicolumn{1}{l}{60} & \multicolumn{1}{l}{60} \\
		\hline 512 & 4 & Mean & {0.11005\%} & 0.28175\% & {0.41681\%} & 1.43813\% \\
		& & Std. Dev. & {0.24153\%} & 0.36669\% & {0.40287\%} & 0.92611\% \\
		& & $n$ & \multicolumn{1}{l}{60} & \multicolumn{1}{l}{60} & \multicolumn{1}{l}{60} & \multicolumn{1}{l}{60} \\
		\hhline{~------} & 8 & Mean & {0.05810\%} & 0.12283\% & {0.33781\%} & 0.83951\% \\
		& & Std. Dev. & {0.19562\%} & 0.33859\% & {0.51161\%} & 0.96077\% \\
		& & $n$ & \multicolumn{1}{l}{60} & \multicolumn{1}{l}{60} & \multicolumn{1}{l}{60} & \multicolumn{1}{l}{60} \\
		\hline
	\end{tabular}
	\caption{Relative error for LP relaxation in {\textsc{Unstructured-Type}} trials. Columns indicate percentage of arcs acting as parents.}
	\label{table:lpgaparc}
\end{table}

\begin{table}[p]
	\centering
	\begin{tabular}{l l l r r r r}
		\hline \multicolumn{7}{l}{\textbf{{\textsc{Structured-Type}}, {\mbox{$\textsc{RR-Child}\mathbf{(0.00)}$}}}} \\
		Nodes & Arcs/Node & & \multicolumn{1}{l}{2\%} & \multicolumn{1}{l}{5\%} & \multicolumn{1}{l}{10\%} & \multicolumn{1}{l}{15\%} \\
		\hline 256 & 4 & Mean & 0.08396\% & 0.12902\% & 0.30359\% & 0.47040\% \\
		& & Std. Dev. & 0.45237\% & 0.30627\% & 0.69299\% & 0.91533\% \\
		& & $n$ & \multicolumn{1}{l}{60} & \multicolumn{1}{l}{60} & \multicolumn{1}{l}{60} & \multicolumn{1}{l}{60} \\
		\hhline{~------} & 8 & Mean & 0.04560\% & 0.05072\% & 0.12079\% & 0.20614\% \\
		& & Std. Dev. & 0.20377\% & 0.19954\% & 0.39895\% & 0.64658\% \\
		& & $n$ & \multicolumn{1}{l}{60} & \multicolumn{1}{l}{60} & \multicolumn{1}{l}{60} & \multicolumn{1}{l}{60} \\
		\hline 512 & 4 & Mean & 0.11863\% & 0.15723\% & 0.22745\% & 0.53016\% \\
		& & Std. Dev. & 0.33121\% & 0.36941\% & 0.42859\% & 0.67835\% \\
		& & $n$ & \multicolumn{1}{l}{60} & \multicolumn{1}{l}{60} & \multicolumn{1}{l}{59} & \multicolumn{1}{l}{60} \\
		\hhline{~------} & 8 & Mean & 0.05038\% & 0.09224\% & 0.20583\% & 0.48200\% \\
		& & Std. Dev. & 0.14546\% & 0.29561\% & 0.47322\% & 0.70649\% \\
		& & $n$ & \multicolumn{1}{l}{60} & \multicolumn{1}{l}{60} & \multicolumn{1}{l}{60} & \multicolumn{1}{l}{60} \\
		\hline
	\end{tabular}
	\caption{Mean and standard deviation of relative error for the {\mbox{$\textsc{RR-Child}(0.00)$}} scheme solutions over all {\textsc{Structured-Type}} computational trials. Note that the statistics for 512 nodes, 4 arcs per node, and 10\% of sinks interdependent were calculated based only on the 59 successful trials.}
	\label{table:rrc0gapnode}
\end{table}

\begin{table}[p]
	\centering
	\begin{tabular}{l l l r r r r}
		\hline \multicolumn{7}{l}{\textbf{{\textsc{Structured-Type}}, {\mbox{$\textsc{RR-Child}\mathbf{(0.01)}$}}}} \\
		Nodes & Arcs/Node & & \multicolumn{1}{l}{2\%} & \multicolumn{1}{l}{5\%} & \multicolumn{1}{l}{10\%} & \multicolumn{1}{l}{15\%} \\
		\hline 256 & 4 & Mean & 0.15995\% & 0.17999\% & 0.42908\% & 0.51373\% \\
		& & Std. Dev. & 0.73355\% & 0.50584\% & 1.06876\% & 0.93875\% \\
		& & $n$ & \multicolumn{1}{l}{60} & \multicolumn{1}{l}{60} & \multicolumn{1}{l}{60} & \multicolumn{1}{l}{60} \\
		\hhline{~------} & 8 & Mean & 0.08235\% & 0.05400\% & 0.16980\% & 0.20614\% \\
		& & Std. Dev. & 0.34519\% & 0.20031\% & 0.53964\% & 0.64658\% \\
		& & $n$ & \multicolumn{1}{l}{60} & \multicolumn{1}{l}{60} & \multicolumn{1}{l}{60} & \multicolumn{1}{l}{60} \\
		\hline 512 & 4 & Mean & 0.12555\% & 0.16074\% & 0.24860\% & 0.61767\% \\
		& & Std. Dev. & 0.33302\% & 0.36889\% & 0.43587\% & 0.89859\% \\
		& & $n$ & \multicolumn{1}{l}{60} & \multicolumn{1}{l}{60} & \multicolumn{1}{l}{60} & \multicolumn{1}{l}{60} \\
		\hhline{~------} & 8 & Mean & 0.08089\% & 0.10822\% & 0.22707\% & 0.51905\% \\
		& & Std. Dev. & 0.25132\% & 0.32182\% & 0.47803\% & 0.75221\% \\
		& & $n$ & \multicolumn{1}{l}{60} & \multicolumn{1}{l}{60} & \multicolumn{1}{l}{60} & \multicolumn{1}{l}{60} \\
		\hline
	\end{tabular}
	\caption{Mean and standard deviation of relative error for the {\mbox{$\textsc{RR-Child}(0.01)$}} scheme solutions over all {\textsc{Structured-Type}} computational trials.}
	\label{table:rrc1gapnode}
\end{table}

\begin{table}[p]
	\centering
	\begin{tabular}{l l l r r r r}
		\hline \multicolumn{7}{l}{\textbf{{\textsc{Structured-Type}}, {\mbox{$\textsc{RR-Child}\mathbf{(0.05)}$}}}} \\
		Nodes & Arcs/Node & & \multicolumn{1}{l}{2\%} & \multicolumn{1}{l}{5\%} & \multicolumn{1}{l}{10\%} & \multicolumn{1}{l}{15\%} \\
		\hline 256 & 4 & Mean & 0.15995\% & 0.27764\% & 0.83326\% & 0.99746\% \\
		& & Std. Dev. & 0.73355\% & 0.70453\% & 1.78568\% & 1.82337\% \\
		& & $n$ & \multicolumn{1}{l}{60} & \multicolumn{1}{l}{60} & \multicolumn{1}{l}{60} & \multicolumn{1}{l}{60} \\
		\hhline{~------} & 8 & Mean & 0.17594\% & 0.14096\% & 0.24590\% & 0.26612\% \\
		& & Std. Dev. & 0.64976\% & 0.53061\% & 0.62247\% & 0.70883\% \\
		& & $n$ & \multicolumn{1}{l}{60} & \multicolumn{1}{l}{60} & \multicolumn{1}{l}{60} & \multicolumn{1}{l}{60} \\
		\hline 512 & 4 & Mean & 0.29846\% & 0.32916\% & 0.51934\% & 1.02361\% \\
		& & Std. Dev. & 0.63220\% & 0.63879\% & 0.77623\% & 1.32727\% \\
		& & $n$ & \multicolumn{1}{l}{60} & \multicolumn{1}{l}{60} & \multicolumn{1}{l}{60} & \multicolumn{1}{l}{60} \\
		\hhline{~------} & 8 & Mean & 0.15358\% & 0.17735\% & 0.40364\% & 0.63427\% \\
		& & Std. Dev. & 0.39285\% & 0.43903\% & 0.63889\% & 0.89402\% \\
		& & $n$ & \multicolumn{1}{l}{60} & \multicolumn{1}{l}{60} & \multicolumn{1}{l}{60} & \multicolumn{1}{l}{60} \\
		\hline
	\end{tabular}
	\caption{Mean and standard deviation of relative error for the {\mbox{$\textsc{RR-Child}(0.05)$}} scheme solutions over all {\textsc{Structured-Type}} computational trials.}
	\label{table:rrc5gapnode}
\end{table}

\begin{table}[p]
	\centering
	\begin{tabular}{l l l r r r r}
		\hline \multicolumn{7}{l}{\textbf{{\textsc{Structured-Type}}, {\mbox{$\textsc{RR-Fair}$}}}} \\
		Nodes & Arcs/Node & & \multicolumn{1}{l}{2\%} & \multicolumn{1}{l}{5\%} & \multicolumn{1}{l}{10\%} & \multicolumn{1}{l}{15\%} \\
		\hline 256 & 4 & Mean & 1.69196\% & 1.68475\% & 3.71021\% & 5.60403\% \\
		& & Std. Dev. & 2.51891\% & 1.88534\% & 3.31362\% & 4.90118\% \\
		& & $n$ & \multicolumn{1}{l}{60} & \multicolumn{1}{l}{60} & \multicolumn{1}{l}{60} & \multicolumn{1}{l}{60} \\
		\hhline{~------} & 8 & Mean & 1.30573\% & 1.33120\% & 2.03500\% & 2.76798\% \\
		& & Std. Dev. & 2.21282\% & 1.83162\% & 2.67148\% & 3.29595\% \\
		& & $n$ & \multicolumn{1}{l}{60} & \multicolumn{1}{l}{60} & \multicolumn{1}{l}{60} & \multicolumn{1}{l}{60} \\
		\hline 512 & 4 & Mean & 0.93481\% & 1.80408\% & 3.24969\% & 5.04507\% \\
		& & Std. Dev. & 0.98142\% & 1.50691\% & 2.20683\% & 3.20177\% \\
		& & $n$ & \multicolumn{1}{l}{60} & \multicolumn{1}{l}{60} & \multicolumn{1}{l}{60} & \multicolumn{1}{l}{60} \\
		\hhline{~------} & 8 & Mean & 0.92729\% & 1.27037\% & 1.27321\% & 1.92335\% \\
		& & Std. Dev. & 1.00769\% & 1.40316\% & 1.14424\% & 1.81029\% \\
		& & $n$ & \multicolumn{1}{l}{60} & \multicolumn{1}{l}{60} & \multicolumn{1}{l}{60} & \multicolumn{1}{l}{60} \\
		\hline
	\end{tabular}
	\caption{Mean and standard deviation of relative error for the {\mbox{$\textsc{RR-Fair}$}} scheme solutions over all {\textsc{Structured-Type}} computational trials.}
	\label{table:rrfgapnode}
\end{table}

\FloatBarrier
\newpage

\begin{table}[p]
	\centering
	\begin{tabular}{l l r r r r}
		\hline \multicolumn{6}{l}{\textbf{{\textsc{Unstructured-Type}}, {\mbox{$\textsc{RR-Child}\mathbf{(0.00)}$}}}} \\
		Nodes & Arcs/Node & \multicolumn{1}{l}{1\%} & \multicolumn{1}{l}{2\%} & \multicolumn{1}{l}{5\%} & \multicolumn{1}{l}{10\%} \\
		\hline 256 & 4 & {3.3\%} & 1.7\% & 0.0\% & 8.3\% \\
		\hhline{~-----} & 8 & 0.0\% & 0.0\% & 0.0\% & 5.0\% \\
		\hline 512 & 4 & 0.0\% & 1.7\% & {6.7\%} & 15.0\% \\
		\hhline{~-----} & 8 & {1.7\%} & 1.7\% & {1.7\%} & 1.7\% \\
		\hline
	\end{tabular}
	\caption{Failure rates for {\mbox{$\textsc{RR-Child}(0.00)$}} in {\textsc{Unstructured-Type}} trials. Percentages indicate the fraction of trials for which the randomized rounding scheme was unable to obtain a feasible solution within 1000 iterations.}
	\label{table:rrc0fail}
\end{table}

\begin{table}[p]
	\centering
	\begin{tabular}{l l r r r r}
		\hline \multicolumn{6}{l}{\textbf{{\textsc{Unstructured-Type}}, {\mbox{$\textsc{RR-Child}\mathbf{(0.01)}$}}}} \\
		Nodes & Arcs/Node & \multicolumn{1}{l}{1\%} & \multicolumn{1}{l}{2\%} & \multicolumn{1}{l}{5\%} & \multicolumn{1}{l}{10\%} \\
		\hline 256 & 4 & {3.3\%} & 0.0\% & 0.0\% & 3.3\% \\
		\hhline{~-----} & 8 & 0.0\% & 0.0\% & 0.0\% & 0.0\% \\
		\hline 512 & 4 & 0.0\% & 0.0\% & 0.0\% & 5.0\% \\
		\hhline{~-----} & 8 & {1.7\%} & 0.0\% & 0.0\% & 1.7\% \\
		\hline
	\end{tabular}
	\caption{Failure rates for {\mbox{$\textsc{RR-Child}(0.01)$}} in {\textsc{Unstructured-Type}} trials. Percentages indicate the fraction of trials for which the randomized rounding scheme was unable to obtain a feasible solution within 1000 iterations.}
	\label{table:rrc1fail}
\end{table}

\begin{table}[p]
	\centering
	\begin{tabular}{l l r r r r}
		\hline \multicolumn{6}{l}{\textbf{{\textsc{Unstructured-Type}}, {\mbox{$\textsc{RR-Child}\mathbf{(0.05)}$}}}} \\
		Nodes & Arcs/Node & \multicolumn{1}{l}{1\%} & \multicolumn{1}{l}{2\%} & \multicolumn{1}{l}{5\%} & \multicolumn{1}{l}{10\%} \\
		\hline 256 & 4 & {3.3\%} & 0.0\% & 0.0\% & 1.7\% \\
		\hhline{~-----} & 8 & 0.0\% & 0.0\% & 0.0\% & 0.0\% \\
		\hline 512 & 4 & 0.0\% & 0.0\% & 0.0\% & 1.7\% \\
		\hhline{~-----} & 8 & {1.7\%} & 0.0\% & 0.0\% & 1.7\% \\
		\hline
	\end{tabular}
	\caption{Failure rates for {\mbox{$\textsc{RR-Child}(0.05)$}} in {\textsc{Unstructured-Type}} trials. Percentages indicate the fraction of trials for which the randomized rounding scheme was unable to obtain a feasible solution within 1000 iterations.}
	\label{table:rrc5fail}
\end{table}

\begin{table}[p]
	\centering
	\begin{tabular}{l l r r r r}
		\hline \multicolumn{6}{l}{\textbf{{\textsc{Unstructured-Type}}, {\mbox{$\textsc{RR-Fair}$}}}} \\
		Nodes & Arcs/Node & \multicolumn{1}{l}{1\%} & \multicolumn{1}{l}{2\%} & \multicolumn{1}{l}{5\%} & \multicolumn{1}{l}{10\%} \\
		\hline 256 & 4 & {3.3\%} & 0.0\% & {11.7\%} & 96.7\% \\
		\hhline{~-----} & 8 & 0.0\% & 8.3\% & {86.7\%} & 100.0\% \\
		\hline 512 & 4 & 0.0\% & 11.7\% & {91.7\%} & 100.0\% \\
		\hhline{~-----} & 8 & {11.7\%} & 71.7\% & 100.0\% & 100.0\% \\
		\hline
	\end{tabular}
	\caption{Failure rates for {\mbox{$\textsc{RR-Fair}$}} in {\textsc{Unstructured-Type}} trials. Percentages indicate the fraction of trials for which the randomized rounding scheme was unable to obtain a feasible solution within 1000 iterations.}
	\label{table:rrffail}
\end{table}

\FloatBarrier
\newpage

\begin{table}[p]
	\centering
	\begin{tabular}{l l l r r r r}
		\hline \multicolumn{7}{l}{\textbf{{\textsc{Unstructured-Type}}, {\mbox{$\textsc{RR-Child}\mathbf{(0.00)}$}}}} \\
		Nodes & Arcs/Node & & \multicolumn{1}{l}{1\%} & \multicolumn{1}{l}{2\%} & \multicolumn{1}{l}{5\%} & \multicolumn{1}{l}{10\%} \\
		\hline 256 & 4 & Mean & {0.08042\%} & 0.09553\% & {0.25293\%} & 0.78186\% \\
		& & Std. Dev. & {0.28710\%} & 0.36233\% & {0.52748\%} & 1.10045\% \\
		& & $n$ & \multicolumn{1}{l}{{58}} & \multicolumn{1}{l}{59} & \multicolumn{1}{l}{60} & \multicolumn{1}{l}{55} \\
		\hhline{~------} & 8 & Mean & {0.01565\%} & 0.08292\% & {0.32456\%} & 0.17906\% \\
		& & Std. Dev. & {0.07513\%} & 0.34117\% & {1.29113\%} & 0.49942\% \\
		& & $n$ & \multicolumn{1}{l}{60} & \multicolumn{1}{l}{60} & \multicolumn{1}{l}{60} & \multicolumn{1}{l}{57} \\
		\hline 512 & 4 & Mean & {0.07700\%} & 0.10265\% & {0.20542\%} & 0.49078\% \\
		& & Std. Dev. & {0.26175\%} & 0.25362\% & {0.45111\%} & 0.57372\% \\
		& & $n$ & \multicolumn{1}{l}{60} & \multicolumn{1}{l}{59} & \multicolumn{1}{l}{{56}} & \multicolumn{1}{l}{51} \\
		\hhline{~------} & 8 & Mean & {0.45111\%} & 0.09384\% & {0.09531\%} & 0.34873\% \\
		& & Std. Dev. & {0.08387\%} & 0.30559\% & {0.21521\%} & 0.60217\% \\
		& & $n$ & \multicolumn{1}{l}{{59}} & \multicolumn{1}{l}{59} & \multicolumn{1}{l}{{59}} & \multicolumn{1}{l}{59} \\
		\hline
	\end{tabular}
	\caption{Mean and standard deviation of relative error for {\mbox{$\textsc{RR-Child}(0.00)$}} scheme solutions over all successful {\textsc{Unstructured-Type}} computational trials.}
	\label{table:rrc0gaparc}
\end{table}

\begin{table}[p]
	\centering
	\begin{tabular}{l l l r r r r}
		\hline \multicolumn{7}{l}{\textbf{{\textsc{Unstructured-Type}}, {\mbox{$\textsc{RR-Child}\mathbf{(0.01)}$}}}} \\
		Nodes & Arcs/Node & & \multicolumn{1}{l}{1\%} & \multicolumn{1}{l}{2\%} & \multicolumn{1}{l}{5\%} & \multicolumn{1}{l}{10\%} \\
		\hline 256 & 4 & Mean & {0.14013\%} & 0.27409\% & {0.52975\%} & 1.50022\% \\
		& & Std. Dev. & {0.52865\%} & 0.72796\% & {1.22847\%} & 2.16678\% \\
		& & $n$ & \multicolumn{1}{l}{{58}} & \multicolumn{1}{l}{60} & \multicolumn{1}{l}{60} & \multicolumn{1}{l}{58} \\
		\hhline{~------} & 8 & Mean & {0.61139\%} & 1.09969\% & {2.04546\%} & 3.54538\% \\
		& & Std. Dev. & {1.79169\%} & 2.93195\% & {4.06981\%} & 5.15958\% \\
		& & $n$ & \multicolumn{1}{l}{60} & \multicolumn{1}{l}{60} & \multicolumn{1}{l}{60} & \multicolumn{1}{l}{60} \\
		\hline 512 & 4 & Mean & {0.14450\%} & 0.35749\% & {0.51950\%} & 1.65035\% \\
		& & Std. Dev. & {0.38393\%} & 0.92922\% & {1.03785\%} & 2.41296\% \\
		& & $n$ & \multicolumn{1}{l}{60} & \multicolumn{1}{l}{60} & \multicolumn{1}{l}{60} & \multicolumn{1}{l}{57} \\
		\hhline{~------} & 8 & Mean & {0.49411\%} & 0.71963\% & {1.44727\%} & 4.06781\% \\
		& & Std. Dev. & {1.42606\%} & 1.38286\% & {2.13480\%} & 3.92718\% \\
		& & $n$ & \multicolumn{1}{l}{{59}} & \multicolumn{1}{l}{60} & \multicolumn{1}{l}{60} & \multicolumn{1}{l}{59} \\
		\hline
	\end{tabular}
	\caption{Mean and standard deviation of relative error for {\mbox{$\textsc{RR-Child}(0.01)$}} scheme solutions over all successful {\textsc{Unstructured-Type}} computational trials.}
	\label{table:rrc1gaparc}
\end{table}

\begin{table}[p]
	\centering
	\begin{tabular}{l l l r r r r}
		\hline \multicolumn{7}{l}{\textbf{{\textsc{Unstructured-Type}}, {\mbox{$\textsc{RR-Child}\mathbf{(0.05)}$}}}} \\
		Nodes & Arcs/Node & & \multicolumn{1}{l}{1\%} & \multicolumn{1}{l}{2\%} & \multicolumn{1}{l}{5\%} & \multicolumn{1}{l}{10\%} \\
		\hline 256 & 4 & Mean & {0.57540\%} & 1.33160\% & {2.41167\%} & 5.33097\% \\
		& & Std. Dev. & {1.44484\%} & 2.25426\% & {3.64680\%} & 5.02341\% \\
		& & $n$ & \multicolumn{1}{l}{{58}} & \multicolumn{1}{l}{60} & \multicolumn{1}{l}{60} & \multicolumn{1}{l}{59} \\
		\hhline{~------} & 8 & Mean & {1.72941\%} & 3.44290\% & {8.96035\%} & 16.80720\% \\
		& & Std. Dev. & {2.91170\%} & 4.67988\% & {8.84133\%} & 10.32456\% \\
		& & $n$ & \multicolumn{1}{l}{60} & \multicolumn{1}{l}{60} & \multicolumn{1}{l}{60} & \multicolumn{1}{l}{60} \\
		\hline 512 & 4 & Mean & {0.39119\%} & 1.13194\% & {2.18552\%} & 4.94240\% \\
		& & Std. Dev. & {0.85671\%} & 1.66150\% & {2.41066\%} & 3.47922\% \\
		& & $n$ & \multicolumn{1}{l}{60} & \multicolumn{1}{l}{60} & \multicolumn{1}{l}{60} & \multicolumn{1}{l}{59} \\
		\hhline{~------} & 8 & Mean & {1.57806\%} & 3.87011\% & {9.29088\%} & 15.76693\% \\
		& & Std. Dev. & {2.56120\%} & 3.59570\% & {6.51458\%} & 8.10469\% \\
		& & $n$ & \multicolumn{1}{l}{{59}} & \multicolumn{1}{l}{60} & \multicolumn{1}{l}{60} & \multicolumn{1}{l}{59} \\
		\hline
	\end{tabular}
	\caption{Mean and standard deviation of relative error for {\mbox{$\textsc{RR-Child}(0.05)$}} scheme solutions over all successful {\textsc{Unstructured-Type}} computational trials.}
	\label{table:rrc5gaparc}
\end{table}

\begin{table}[p]
	\centering
	\begin{tabular}{l l l r r r r}
		\hline \multicolumn{7}{l}{\textbf{{\textsc{Unstructured-Type}}, {\mbox{$\textsc{RR-Fair}$}}}} \\
		Nodes & Arcs/Node & & \multicolumn{1}{l}{1\%} & \multicolumn{1}{l}{2\%} & \multicolumn{1}{l}{5\%} & \multicolumn{1}{l}{10\%} \\
		\hline 256 & 4 & Mean & {6.25055\%} & 12.06161\% & {22.86896\%} & 57.94037\% \\
		& & Std. Dev. & {6.03445\%} & 7.11377\% & {8.44746\%} & 0.78205\% \\
		& & $n$ & \multicolumn{1}{l}{{58}} & \multicolumn{1}{l}{60} & \multicolumn{1}{l}{{53}} & \multicolumn{1}{l}{2} \\
		\hhline{~------} & 8 & Mean & {19.62852\%} & 38.37295\% & {83.90358\%} & \multicolumn{1}{c}{--} \\
		& & Std. Dev. & {11.38919\%} & 18.65668\% & {15.18412\%} & \multicolumn{1}{c}{--} \\
		& & $n$ & \multicolumn{1}{l}{60} & \multicolumn{1}{l}{55} & \multicolumn{1}{l}{{8}} & \multicolumn{1}{l}{0} \\
		\hline 512 & 4 & Mean & {6.29454\%} & 10.59501\% & {22.63133\%} & \multicolumn{1}{c}{--} \\
		& & Std. Dev. & {3.56471\%} & 5.17104\% & {8.17028\%} & \multicolumn{1}{c}{--} \\
		& & $n$ & \multicolumn{1}{l}{60} & \multicolumn{1}{l}{53} & \multicolumn{1}{l}{{5}} & \multicolumn{1}{l}{0} \\
		\hhline{~------} & 8 & Mean & {19.80782\%} & 37.17727\% & \multicolumn{1}{c}{--} & \multicolumn{1}{c}{--} \\
		& & Std. Dev. & {7.31611\%} & 11.84658\% & \multicolumn{1}{c}{--} & \multicolumn{1}{c}{--} \\
		& & $n$ & \multicolumn{1}{l}{{53}} & \multicolumn{1}{l}{17} & \multicolumn{1}{l}{0} & \multicolumn{1}{l}{0} \\
		\hline
	\end{tabular}
	\caption{Mean and standard deviation of relative error for {\mbox{$\textsc{RR-Fair}$}} scheme solutions over all successful {\textsc{Unstructured-Type}} computational trials. Dashes indicate trial sets for which no feasible solutions could be found.}
	\label{table:rrfgaparc}
\end{table}

\end{document}